\documentclass{amsart}
\usepackage[utf8]{inputenc}

\usepackage{amsmath}
\usepackage{amssymb}
\usepackage{verbatim}
\usepackage{graphicx}

\usepackage{tikz-cd}
\usepackage{geometry}
\usepackage{esint}
\usepackage{soul}
\usepackage[colorlinks=true,urlcolor=blue,citecolor=red,linkcolor=blue,linktocpage,pdfpagelabels,bookmarksnumbered,bookmarksopen]{hyperref}

\numberwithin{equation}{section}

\usepackage[hyperpageref]{backref}

\subjclass[2010]{46E35, 35J92, 35P30}
\keywords{Sobolev embeddings, $p-$Laplacian, Lane-Emden equation, inradius, distance function, capacity.}

\date{\today}
\newtheorem{thm}{Theorem}[section]
\newtheorem{cor}[thm]{Corollary}
\newtheorem{prop}[thm]{Proposition}
\newtheorem{lemma}[thm]{Lemma}
\theoremstyle{definition}
\newtheorem{defn}[thm]{Definition}
\newtheorem{ex}[thm]{Example}
\newtheorem{rem}[thm]{Remark}
\newtheorem*{ack}{Acknowledgments}

\title{Sobolev embeddings and distance functions}

\author[Brasco]{Lorenzo Brasco}
\address[L.\ Brasco]{Dipartimento di Matematica e Informatica
	\newline\indent
	Universit\`a degli Studi di Ferrara
	\newline\indent
	Via Machiavelli 35, 44121 Ferrara, Italy}
\email{lorenzo.brasco@unife.it}

\author[Prinari]{Francesca Prinari}

\address[F. Prinari]{Dipartimento di Scienze Agrarie, Alimentari e Agro-ambientali
\newline\indent 
Universit\`a di Pisa
\newline\indent
Via del Borghetto 80, 56124 Pisa, Italy}
\email{francesca.prinari@unipi.it}

\author[Zagati]{Anna Chiara Zagati}
\address[A.\,C.\ Zagati]{Dipartimento di Scienze Matematiche, Fisiche e Informatiche
	\newline\indent
	Universit\`a di Parma
	\newline\indent
	Parco Area delle Scienze 53/a, Campus, 43124 Parma, Italy}
	\email{annachiara.zagati@unipr.it}

\begin{document}

\begin{abstract}
On a general open set of the euclidean space, we study the relation between the embedding of the homogeneous Sobolev space $\mathcal{D}^{1,p}_0$ into $L^q$ and the summability properties of the distance function. We prove that in the superconformal case (i.e. when $p$ is larger than the dimension) these two facts are equivalent, while in the subconformal and conformal cases (i.e. when $p$ is less than or equal to the dimension) we construct counterexamples to this equivalence.
In turn, our analysis permits to study the asymptotic behaviour of the positive solution of the Lane-Emden equation for the $p-$Laplacian with sub-homogeneous right-hand side, as the exponent $p$ diverges to $\infty$. The case of first eigenfunctions of the $p-$Laplacian is included, as well. As particular cases of our analysis, we retrieve some well-known convergence results, under optimal assumptions on the open sets. We also give some new geometric estimates for generalized principal frequencies.
\end{abstract}

	\maketitle
	
	\begin{center}
		\begin{minipage}{11cm}
			\small
			\tableofcontents
		\end{minipage}
	\end{center}
	
	\section{Introduction}
	
	\subsection{Homogeneous Sobolev spaces}
	Let $1<p<\infty$ and let $\Omega\subseteq\mathbb{R}^N$ be an open set. We indicate by $\mathcal{D}^{1,p}_0(\Omega)$ the {\it homogeneous Sobolev space} defined by the completion of $C^\infty_0(\Omega)$, with respect to the norm
	\[
	\psi\mapsto \|\nabla \psi\|_{L^p(\Omega)},\qquad \mbox{ for every } \psi\in C^\infty_0(\Omega).
	\]
Our primary goal is to deepen the study of conditions on $\Omega$ assuring the validity of the continuous embedding
\begin{equation}
\label{embedding}
\mathcal{D}^{1,p}_0(\Omega)\hookrightarrow L^q(\Omega),
\end{equation}
in the range $1\le q\le p$. Equivalently, if we introduce the {\it generalized principal frequencies}
	\begin{equation}\label{lapq}
		\lambda_{p,q}(\Omega) := \inf_{\psi \in C^{\infty}_0(\Omega)} \left\{\int_{\Omega} |\nabla \psi|^p \, dx\, :\,\|\psi\|_{L^q(\Omega)}=1\right\},
	\end{equation}
	we seek for necessary and sufficient conditions on $\Omega$ assuring that $\lambda_{p,q}(\Omega)>0$. Indeed, $\lambda_{p,q}(\Omega)$ is nothing but the sharp constant in the following Poincar\'e-type inequality
\[
c\,\left(\int_\Omega |\psi|^q\,dx\right)^\frac{p}{q}\le \int_\Omega |\nabla \psi|^p\,dx,\qquad \mbox{ for every } \psi\in C^\infty_0(\Omega).
\]
This is a classical subject, we refer to \cite[Chapter 15]{Maz} for a thorough treatment of the problem. In particular, before proceeding further, it is important to recall some important facts. 
\par
The first one is the following remarkable equivalence
	\[
	\mathcal{D}^{1,p}_0(\Omega)\hookrightarrow L^q(\Omega) \mbox{ is continuous }\qquad \iff \qquad \mathcal{D}^{1,p}_0(\Omega)\hookrightarrow L^q(\Omega) \mbox{ is compact},
	\]
	which holds in the {\it sub-homogeneous case} $1\le q<p<\infty$. We refer to  \cite[Theorem 15.6.2]{Maz} for this result (see also \cite[Theorem 1.2]{BR}, for a different proof). Such an equivalence ceases to be true at the threshold $q=p$, as shown by the simple example of
	\[
	\Omega=\mathbb{R}^{N-1}\times(-1,1).
	\]
A second important fact is that, for the {\it super-homogeneous case} $p<q$, we have
\[
\mathcal{D}^{1,p}_0(\Omega)\hookrightarrow L^p(\Omega) \mbox{ is continuous }\qquad \iff \qquad \mathcal{D}^{1,p}_0(\Omega)\hookrightarrow L^q(\Omega) \mbox{ is continuous},
\]
and
\[
\mathcal{D}^{1,p}_0(\Omega)\hookrightarrow L^p(\Omega) \mbox{ is compact }\qquad \iff \qquad \mathcal{D}^{1,p}_0(\Omega)\hookrightarrow L^q(\Omega) \mbox{ is compact}.
\]
More precisely, here $q$ is such that
\[
p<q\left\{\begin{array}{ll}
<\dfrac{N\,p}{N-p},& \mbox{ if } 1<p<N,\\
<\infty, & \mbox{ if } p=N,\\
\le \infty,& \mbox{ if } p\ge N.
\end{array}
\right.
\]
We refer to \cite[Theorem 15.4.1]{Maz} and \cite[Theorem 15.6.1]{Maz} for these equivalences. This second fact explains why we will essentially limit ourselves to treat the case $q\le p$ and only briefly discuss the case $q>p$.
\vskip.2cm\noindent
In particular, in this paper we want to discuss the link between the continuous (and compact) embedding \eqref{embedding} and the summability of the distance function
	\[ 
	d_{\Omega}(x):=\min_{y \in \partial\Omega} |x-y|, \qquad \mbox{ for every } x \in \Omega.
	\]
At this aim, it is useful to recall that in \cite{BR} (see also \cite{BB}), a similar study has been done, with the so-called {\it $p-$torsion function} of $\Omega$ in place of $d_\Omega$. The former, denoted by $w^{\Omega}_{p,1}$, is formally defined as the positive solution of 
	\[ 
	-\Delta_p u = 1, \ \mbox{ in } \Omega, \qquad u=0,\ \mbox{ on }\partial\Omega,
	\]
	where $\Delta_p u=\mathrm{div}(|\nabla u|^{p-2} \nabla u)$ is the $p-$Laplace operator, see \cite[Definitions 2.1 and 2.2]{BR} for the precise definition.
Then, for the case $q<p$,	 in \cite[Theorem 1.2]{BR} the first author and Ruffini proved that 
\[
\mathcal{D}^{1,p}_0(\Omega)\hookrightarrow L^q(\Omega) \mbox{ is continuous }\qquad \iff \qquad w^\Omega_{p,1} \in L^{\frac{p-1}{p-q}\,q}(\Omega).
\]
As for the limit case $q=p$, by \cite[Theorem 1.3]{BR} we have
\[
\mathcal{D}^{1,p}_0(\Omega)\hookrightarrow L^p(\Omega) \mbox{ is continuous }\qquad \iff \qquad w^\Omega_{p,1} \in L^\infty(\Omega),
\]
and
\[
\mathcal{D}^{1,p}_0(\Omega)\hookrightarrow L^p(\Omega) \mbox{ is compact }\qquad \iff \qquad \lim_{R\to+\infty}\|w^\Omega_{p,1}\|_{L^\infty(\mathbb{R}^N\setminus B_R)}=0,
\]
where $B_R$ is the $N-$dimensional ball of center $0$ and radius $R>0$.
\par
These characterizations are quite useful: we give an example in Appendix \ref{sec:A} of an open planar set with infinite volume, for which one can get the compactness of the embedding $\mathcal{D}^{1,2}_0(\Omega)\hookrightarrow L^2(\Omega)$ by appealing to the torsion function. Nevertheless, dealing with the $p-$torsion function is not always practical, thus the previous characterizations are a bit implicit. It would be desirable to have some characterizations which are more intrinsically geometric. 
\par
Roughly speaking, in this paper we will study to which extent we can replace $w^\Omega_{p,1}$ with $d_\Omega$ in the aforementioned results. 
For capacitary reasons, we will see that this is possible only in the {\it superconformal case} $p>N$. In the case $p\le N$, we will find that some summability conditions on $d_\Omega$ are {\it necessary} for the embeddings to hold, but not sufficient.

	\subsection{Main results}
	
In this paper, we will show at first that for $1\le q<p$
\[
\mathcal{D}^{1,p}_0(\Omega)\hookrightarrow L^q(\Omega) \mbox{ is continuous } \qquad \Longrightarrow\qquad d_{\Omega} \in L^{\frac{p\,q}{p-q}}(\Omega).
\]
In the case $q=p$, we will prove that
\[
\mathcal{D}^{1,p}_0(\Omega)\hookrightarrow L^p(\Omega) \mbox{ is continuous } \qquad \Longrightarrow\qquad d_{\Omega} \in L^\infty(\Omega),
\]
and
\[
\mathcal{D}^{1,p}_0(\Omega)\hookrightarrow L^p(\Omega) \mbox{ is compact } \qquad \Longrightarrow\qquad \lim_{R\to +\infty} \|d_\Omega\|_{L^\infty(\mathbb{R}^N\setminus B_R)}=0.
\]
Moreover, when $p>N$ we will show that {\it all these implications actually become equivalences}. We will also construct suitable counterexamples to show that for $1<p\le N$, on the contrary, none of the converse implications hold true.
\par
We recall that the condition
\[
\lim_{R\to +\infty} \|d_\Omega\|_{L^\infty(\mathbb{R}^N\setminus B_R)}=0,
\]
is somehow classical in the theory of Sobolev spaces, when it holds $\Omega$ is said to be {\it quasibounded} (see for example \cite{AF, Cl}).
\begin{rem}[Comparison with previous results, $q<p$]
The fact that the continuity of the embedding implies the stated summability of $d_\Omega$ can be found in \cite[Theorem 3]{vBe}, at least for the case $p=2$. However, apart for the generalization to $p\not=2$, our proof is different and exploits a comparison principle for the sub-homogeneous Lane-Emden equation, that we recently proved in \cite{BPZ}. This in turn permits to obtain a geometric estimate on $\lambda_{p,q}$ of the type 
\[
\lambda_{p,q}(\Omega) \,\left(\displaystyle\int_{\Omega} d_{\Omega}^{\frac{p\,q}{p-q}} \, dx \right)^{\frac{p-q}{q}}\le C,
\]
with a constant $C$ which does not blow-up as $q\nearrow p$, differently from \cite{vBe} (see Remark \ref{rem:vanbasten} below).
\par
The converse implication exploits the same idea of \cite{vBe}, i.e. it is based on Hardy's inequality
\[
C_\Omega\,\int_\Omega \frac{|u|^p}{d_\Omega^p}\,dx\le \int_\Omega |\nabla u|^p\,dx,\qquad \mbox{ for every } u\in C^\infty_0(\Omega).
\]
There is however a difference here: while the result in \cite{vBe} is proved conditionally on the validity of such an inequality, we use here the important fact that such an inequality holds {\it for every} open set $\Omega\subsetneq \mathbb{R}^N$, provided $p>N$. This is a major result proved independently by Lewis \cite{Lewis} and Wannebo \cite{Wa}. 
\end{rem} 
\begin{rem}[Comparison with previous results, $q=p$]
Here our embedding results were already known and due to Adams. However, to prove that on a quasibounded open set the embedding $\mathcal{D}^{1,p}_0(\Omega)\hookrightarrow L^p(\Omega)$ is compact for $p>N$, we use a different argument. As before, the crucial ingredient is the Hardy inequality recalled above, which permits to give a simpler proof. We believe this fact to be of independent interest.
\end{rem}
With these embedding results at hand, we then proceed to study the asymptotic behaviour of both $\lambda_{p,q}(\Omega)$ and their relevant extremals (provided they exist), in the limit as $p$ goes to $\infty$. Here we are going to unify and extend previous results, scattered in the literature. For example, for the generalized principal frequencies, we will prove that 
\[
\lim_{p\to\infty}\Big( \lambda_{p,q}(\Omega) \Big)^\frac{1}{p}=\frac{1}{\|d_{\Omega}\|_{L^q(\Omega)}}\qquad \mbox{ and }\qquad \lim_{p\to\infty}\Big( \lambda_p(\Omega) \Big)^\frac{1}{p} =\frac{1}{r_{\Omega}},
\]	
for {\it every} open set $\Omega\subsetneq\mathbb{R}^N$. Here we used the simplified notation $\lambda_p(\Omega)=\lambda_{p,p}(\Omega)$. The quantity $r_\Omega$ is the {\it inradius} of $\Omega$, which coincides with the supremum of the distance function.
\par
As for the relevant extremals, we will prove at first that under the assumption $d_\Omega\in L^q(\Omega)$ for some $1\le q<\infty$ and $\Omega$ connected, we have 
\[
\lim_{p\to\infty} \|w^\Omega_{p,q}-d_\Omega\|_{L^q(\Omega)}=0,
\]
where for $q<p$ the function $w^\Omega_{p,q}$ is the unique positive solution of the {\it Lane-Emden equation}
\[
-\Delta_p u=u^{q-1},\qquad \mbox{ in }\Omega,
\]
with Dirichlet homogeneous boundary conditions.
By observing that (see equation \eqref{pqnorm} below)
\[
\|w^\Omega_{p,q}\|_{L^q(\Omega)}=\Big(\lambda_{p,q}(\Omega)\Big)^\frac{1}{q-p},
\]
we get that 
\[
\mathtt{w}_{p,q}^\Omega:=\Big(\lambda_{p,q}(\Omega)\Big)^\frac{1}{p-q}\,w^\Omega_{p,q},
\]
is the unique positive extremal in $\mathcal{D}^{1,p}_0(\Omega)$ for $\lambda_{p,q}(\Omega)$ and we have 
\[
\lim_{p\to\infty} \left\|\mathtt{w}^\Omega_{p,q}-\frac{d_\Omega}{\|d_\Omega\|_{L^q(\Omega)}}\right\|_{L^q(\Omega)}=0.
\]
We remark that this convergence is obtained under the optimal assumption $d_\Omega\in L^q(\Omega)$, thus $\Omega$ is not supposed either bounded or with finite volume. 
Moreover, by an interpolation argument, we can upgrade these convergences and infer convergence in $L^r(\Omega)$ and $C^{0,\beta}(\overline\Omega)$, for every $q<r\le \infty$ and  $0<\beta<1$.
\par
In the case of $\lambda_p(\Omega)$, extremals in $\mathcal{D}^{1,p}_0(\Omega)$ are usually called {\it first eigenfunctions of the $p-$Laplacian with Dirichlet conditions}. In this case, we will prove that if $\Omega$ is connected and quasibounded, then the family $\{u_p\}$ is precompact in $L^{\infty}(\Omega)$ and  every accumulation point $u_{\infty}$ is a solution of 
\[
\min_{u \in W^{1,\infty}(\Omega)} \Big\{\|\nabla u \|_{L^{\infty}(\Omega)} :\, \| u \|_{L^{\infty}(\Omega)}=1, \, u \equiv 0 \mbox{ on } \partial \Omega \Big\}=\frac{1} {r_{\Omega}}.
\]
Here each $u_p$ is the first positive eigenfunction on $\Omega$, normalized by the requirement to have unit $L^p(\Omega)$ norm.
\par
Finally, we will also consider  the variational problem associated to the endpoint Poincar\'e-Sobolev constant $\lambda_{p,\infty}(\Omega)$ for $p>N$,  recently studied in \cite{EP} and \cite{HL} when $\Omega$ is a bounded open set.  More precisely, we will  generalize   \cite[Theorems 2.1 and 2.3]{EP},  by showing  that for {\sl every} open set $\Omega\subsetneq\mathbb{R}^N$, it holds 
\[ 
\lim_{q \to \infty}  \lambda_{p,q}(\Omega) = \lambda_{p,\infty}(\Omega),
\] 
and  by proving that the variational problem defining $ \lambda_{p,\infty}(\Omega)$ has a minimizer under the sole assumption of quasiboundedness on $\Omega$.
Moreover, inspired by \cite[Theorem 3.3]{EP}, we will also study the asymptotic behaviour of both $\lambda_{p,\infty}$ and its extremals, as $p$ goes to $\infty$.
\par
At the level of generalized principal frequencies, we can summarize the previous convergence results with the following diagram, 
which holds for every $N < p < \infty$, every $1\le q\le p$ and every open set $\Omega\subsetneq \mathbb{R}^N$ (see  Corollary \ref{teo:limite}, Corollary \ref{teo:limiteq} and Corollary \ref{teo:asymp-p}).
\begin{figure}[h]
\centering
\begin{tikzcd}[row sep=huge, column sep=large]
\Big(\lambda_{p,q}(\Omega)\Big)^\frac{1}{p} \arrow[r, "q\to\infty"]  \arrow[d, "p\to\infty" ] \arrow[dr, "p=q\to\infty"]
& \Big(\lambda_{p,\infty}(\Omega)\Big)^\frac{1}{p} \arrow[d, "p\to\infty" ] \\
 \dfrac{1}{\|d_{\Omega}\|_{L^q(\Omega)}} 
&  \dfrac{1}{r_{\Omega}}
\end{tikzcd}
\end{figure}
\par
 The elements in the bottom line have to be considered as $0$, whenever $d_\Omega\not\in L^q(\Omega)$ or $r_\Omega=+\infty$. Moreover, when $\Omega$ is such that $d_\Omega\in L^{q_0}(\Omega)$ for some $q_0<\infty$, then we can close the diagram by 
\[
 \dfrac{1}{\|d_{\Omega}\|_{L^q(\Omega)}} \stackrel{q\to\infty}{\longrightarrow} \frac{1}{r_\Omega},
\]
thus making it commutative.
\begin{rem}[Comparison with previous results]
These convergence results are not a complete novelty, of course. However, our treatment extends, improves and generalizes some known results previously obtained by various authors in some particular cases (i.e. $q=1$ or $q=p$) and under more restrictive assumptions on $\Omega$. 
 The  asymptotics for $ \lambda_{p}(\Omega)$  generalize   \cite[Lemma 1.5]{JLM} and \cite[Theorem 3.1]{FIN}, shown  under the restrictive assumption  that $\Omega$ is  a bounded open set (see also \cite[Theorem 5.1]{CDPJ}). 
 \par
 The behaviour of  $\lambda_{p,1}(\Omega)$  is due to \cite[Proposition 2.1]{BdBM} and \cite[Theorem 1]{Ka}, in the case of an open bounded set, after getting  the uniform convergence of the relevant extremals $w_{p,1}^\Omega$ to $d_{\Omega}$. This result  is extended to sets  with finite volume  in \cite[Corollary A.4]{BBP}. As for the constant $\lambda_{p,\infty}$, the case of open bounded sets is considered in \cite{EP}.
 \end{rem} 
We point out that the extension of these convergence results to open sets which are unbounded or with infinite volume is not trivial: we crucially rely on the comparison principle for the Lane-Emden equation recalled above, as well as on the asymptotic behaviour as $p$ goes to $\infty$ of the following Morrey-type constant
\[ 
\mu_p(B_1):=\inf_{u\in W^{1,p}(B_1)}\left\{\int_{B_1} |\nabla u|^p\,dx\, :\, u(0)=1 \mbox{ and } u(z)=0 \right\}, \quad \mbox{ with } z \in \partial B_1. 
\]
The study of this constant is contained in Section \ref{sez4} below. The study of the asymptotics for $\mu_p(B_1)$ permits in turn to obtain a couple of collateral results, which are interesting in themselves: the behaviour for large $p$ of the two sharp Hardy constants
	\[
	\mathfrak{h}_{p}(\Omega)=\inf_{u\in C^\infty_0(\Omega)}\left\{\int_\Omega |\nabla u|^p\,dx\, :\, \left\|\frac{u}{d_\Omega}\right\|_{L^p(\Omega)}=1\right\},
	\]
and
	\[
	\mathfrak{h}_{p,\infty}(\Omega)=\inf_{u\in C^\infty_0(\Omega)}\left\{\int_\Omega |\nabla u|^p\,dx\, :\, \left\|\frac{u}{d_\Omega^{1-\frac{N}{p}}}\right\|_{L^\infty(\Omega)}=1\right\},
	\]
as well as the asymptotic behaviour of the sharp Morrey constant
\[
	\mathfrak{m}_{p}(\Omega):=\inf_{u\in C_0^{\infty}(\Omega)} \left\{\int_\Omega |\nabla u|^p\,dx\, :\, [u]_{C^{0,\alpha_p}(\overline\Omega)}=1 \right\},\qquad \mbox{ where }\alpha_p:=1-\frac{N}{p}.
\]
We observe that the latter is actually independent of the open set $\Omega$, i.e. it coincides with that of the whole space (see Corollary \ref{cor:morrey} below, for example). Some studies on such a constant and its extremals have been done recently by Hynd and Seuffert in a series of papers (see \cite{HS1, HS2} and \cite{HS3}), but the exact value of the sharp constant is still unknown. We show in Corollary \ref{cor:morrey} that 
\[
\lim_{p\to\infty} \Big(\mathfrak{m}_{p}(\Omega)\Big)^\frac{1}{p}=\lim_{p\to\infty} \Big(\mathfrak{m}_{p}(\mathbb{R}^N)\Big)^\frac{1}{p}=1.
\]	
	\subsection{Plan of the paper}
	In Section \ref{sez2} we present the functional analytic setting, give some preliminary embedding results and recall some important properties of the Lane-Emden equation. 
	\par
	Then in Section \ref{sez3} we discuss the role of summability assumptions on the distance function $d_{\Omega}$. This 
section also contains a generalized version of the classical Ascoli-Arzel\`a Theorem, for continuous functions on quasibounded open sets.	
\par
	In the subsequent Section \ref{sez4} we present one of the key ingredients of this work, i.e. an Hardy's inequality which holds for {\it every} open set $\Omega \subsetneq \mathbb{R}^N$ and $p>N$ (see Theorem \ref{teo:hardyineq}). In particular, we focus on the asymptotics of the relevant sharp constant, when $p$ goes to $\infty$. For this reason, we first need to study the asymptotics of a particular sharp Morrey-type constant (see Lemma \ref{lm:morreyA}).
	\par
Section \ref{sez5} contains the main embedding theorems for $\mathcal{D}^{1,p}_0(\Omega)$ and their relations with the summability of the distance function.
Finally, in Section \ref{Sec:asymp} we investigate the asymptotic behaviour for the generalized principal frequencies and the positive solutions of Lane-Emden equation when $p$ goes to $\infty$.

	\begin{ack}
	We thank Ryan Hynd and Erik Lindgren for some discussions on Hardy's inequality and the constant $\lambda_{p,\infty}$. This paper has been finalized during the meeting ``{\it PDEs in Cogne: a friendly meeting in the snow\,}'', held in Cogne in January 2023. 
We wish to thank the organizers 
for the kind invitation and the nice working atmosphere provided during the staying. 
\par
	 F.\,P. and A.\,C.\,Z. are members of the {\it Gruppo Nazionale per l'Analisi Matematica, la Probabilit\`a
e le loro Applicazioni} (GNAMPA) of the Istituto Nazionale di Alta Matematica (INdAM).
	\end{ack}

	\section{Preliminaries}\label{sez2}
	
	\subsection{Notation}
	For $x_0\in\mathbb{R}^N$ and $R>0$, we will denote by $B_R(x_0)$ the $N-$dimensional open ball with radius $R$, centered at $x_0$. In particular, when the center coincides with the origin, we will simply write $B_R$. For an open set $\Omega\subsetneq\mathbb{R}^N$,  
	we denote by $d_{\Omega}$ the distance function from the boundary $\partial\Omega$, defined by
	\[ 
	d_{\Omega}(x):=\min_{y \in \partial\Omega} |x-y|, \qquad \mbox{ for every } x \in \Omega.
	\]
	We extend $d_{\Omega}$ by $0$ in $\mathbb{R}^N\setminus \Omega.$ 
		We define the {\it inradius} $r_\Omega$ of $\Omega$ as the radius of a largest ball contained in $\Omega$. More precisely, this quantity is given by
	\begin{equation}
		\label{inradius}
		r_\Omega=\sup \Big\{r>0\, :\, \mbox{ there exists }x_0\in \Omega \mbox{ such that } B_r(x_0)\subset\Omega\Big\}.
	\end{equation}
	It is well-known that this coincides with the supremum over $\Omega$ of $d_\Omega$.
	\vskip.2cm \noindent 
	 The following simple result will be quite useful.
	\begin{lemma}
	\label{lm:utile?}
	Let $\Omega\subsetneq\mathbb{R}^N$ be an open set, 
	we denote by $C_0(\Omega)$ the completion of $C^\infty_0(\Omega)$ with respect to the sup norm. Then we get
	\[
	C_0(\Omega)\subset \Big\{u\in C_{\rm bound}(\overline\Omega)\, :\, u=0 \mbox{ on }\partial\Omega\Big\}.
	\]	
	Here $C_{\rm bound}(\overline\Omega)$ is the set of continuous and bounded functions on $\overline\Omega$.
	\end{lemma}
\begin{proof}
Let $\{u_n\}_{n\in\mathbb{N}}\subset C^\infty_0(\Omega)$ be a Cauchy sequence, with respect to the sup norm. In particular, it is Cauchy sequence in $C_{\rm bound}(\overline\Omega)$, which is a Banach space (see for example \cite[Theorem 7.9]{HeSt}). Thus, there exists $u\in C_{\rm bound}(\overline\Omega)$ such that $u_n$ converges to $u$ uniformly on $\overline\Omega$. Moreover, such a function must vanish at the boundary $\partial\Omega$, as a uniform limit of functions with compact support in $\Omega$.
\end{proof}
For every open set $\Omega\subseteq \mathbb{R}^N$, $0<\alpha \le 1$ and $u$ a continuous function on $\overline{\Omega}$, we set
	\[ 
	[u]_{C^{0,\alpha}(\overline{\Omega})}=\sup_{x\ne y;\, x,y\in\overline{\Omega}} \frac{|u(x)-u(y)|}{|x-y|^\alpha}.
		\]
We have the following interpolation--type estimate.
\begin{lemma}
\label{lm:interpol_holder}
Let $0<\beta<\alpha\le 1$, let $1\le \gamma\le \infty$ and let $\Omega\subsetneq\mathbb{R}^N$ be an open set. For every $u\in C_0(\Omega)\cap L^\gamma(\Omega)$ such that 
\[
[u]_{C^{0,\alpha}(\overline\Omega)}<+\infty,
\]
we have
\[
[u]_{C^{0,\beta}(\overline\Omega)}\le C_1\,\|u\|_{L^\gamma(\Omega)}^\theta\,[u]^{1-\theta}_{C^{0,\alpha}(\overline\Omega)},\qquad \mbox{ with } \theta=\frac{\alpha-\beta}{\alpha+\dfrac{N}{\gamma}},
\]
for some $C_1=C_1(N,\alpha,\beta,\gamma)>0$.
Moreover, if $1\le \gamma<\infty$ we also have
\[
\|u\|_{L^\infty(\Omega)}\le C_2\,\|u\|_{L^\gamma(\Omega)}^\chi\,[u]^{1-\chi}_{C^{0,\alpha}(\overline\Omega)},\qquad \mbox{ with } \chi=\frac{\alpha}{\alpha+\dfrac{N}{\gamma}},
\] 
for some $C_2=C_2(N,\alpha,\gamma)>0$. 
\end{lemma}	
\begin{proof}
By Lemma \ref{lm:utile?}, we can extend $u$ to a continuous function on the whole $\mathbb{R}^N$, by setting it to be $0$ on $\mathbb{R}^N\setminus\overline\Omega$. We first observe that for such an extension it holds
	\begin{equation}\label{eq:normeHold} 
			[u]_{C^{0,\alpha}(\overline\Omega)}= [u]_{C^{0,\alpha}(\mathbb{R}^N)}.
		\end{equation}
		Indeed, we may write
		\[ 
		\begin{split}
			[u]_{C^{0,\alpha}(\mathbb{R}^N)}&=\max\left\{ [u]_{C^{0,\alpha}(\overline\Omega)}, \sup_{ x\in \overline{\Omega},y\notin\overline{\Omega}} \frac{|u(x)|}{|x-y|^{\alpha}} \right\}\\
			&= \max\left\{ [u]_{C^{0,\alpha}(\overline\Omega)}, \sup_{ x\in \Omega ,y\notin\overline{\Omega} } \frac{|u(x)|}{|x-y|^{\alpha}} \right\}.
		\end{split}
		\]
		If $x\in\Omega$ and $y\notin\overline{\Omega}$ 
		then the segment $\overline{x\,y}$ connecting $x$ and $y$ is such that $\overline{x\,y}\cap \Omega\ne\emptyset$ and $\overline{x\,y}\cap(\mathbb{R}^N\setminus \overline \Omega)\ne\emptyset.$
		Hence, there exists $y_0\in \overline{x\,y}\cap \partial \Omega$
		satisfying
		\[ 
		|x-y|\ge |x-y_0|. 
		\]
		This implies 
		\[ \frac{|u(x)|}{|x-y|^{\alpha}}\le \frac{|u(x)-u(y_0)|}{|x-y_0|^{\alpha}}\le [u]_{C^{0,\alpha}(\overline\Omega)},\]
		that gives 
		\[ 
		\sup_{x\in \Omega,y\notin\overline{\Omega}} \frac{|u(x)|}{|x-y|^{\alpha}} \le [u]_{C^{0,\alpha}(\overline\Omega)}.
		\]  
		This  concludes the proof of \eqref{eq:normeHold}.
\par
We now come to the proof of the claimed interpolation inequality.
For every $x,y\in\mathbb{R}^N$ such that $x\not=y$, we write
\begin{equation}
\label{freddino}
\begin{split}
\frac{|u(x)-u(y)|}{|x-y|^\beta}&=\left(\frac{|u(x)-u(y)|}{|x-y|^\alpha}\right)^\frac{\beta}{\alpha}\,|u(x)-u(y)|^\frac{\alpha-\beta}{\alpha}\\
&\le [u]_{C^{0,\alpha}(\overline{\Omega})}^\frac{\beta}{\alpha}\,\left(|u(x)|^\frac{\alpha-\beta}{\alpha}+|u(y)|^\frac{\alpha-\beta}{\alpha}\right).
\end{split}
\end{equation}
Observe that we used the triangle inequality and the sub-additivity of concave powers.
In order to estimate the last term, we use that for every $z\in B_R(x)$ we have
\[
\begin{split}
|u(x)|\le |u(x)-u(z)|+|u(z)|&\le [u]_{C^{0,\alpha}(\overline\Omega)}\,|x-z|^{\alpha}+|u(z)|\\
&\le R^{\alpha}\,[u]_{C^{0,\alpha}(\overline\Omega)}+|u(z)|,
\end{split}
\]
where we also used \eqref{eq:normeHold}.
We now take the integral average of this estimate on $B_R(x)$. This gives
\begin{equation}
\label{pazienza}
\begin{split}
|u(x)|&\le R^\alpha\,[u]_{C^{0,\alpha}(\overline\Omega)}+\frac{1}{\omega_N\, R^N}\,\int_{B_R(x)} |u(z)|\,dz\\
&\le R^\alpha\,[u]_{C^{0,\alpha}(\overline\Omega)}+(\omega_N\,R^N)^{-\frac{1}{\gamma}}\,\left(\int_{\Omega} |u(z)|^\gamma\,dz\right)^\frac{1}{\gamma}.
\end{split}
\end{equation}
With the same argument, we have also
\[
|u(y)|\le R^{\alpha}\,[u]_{C^{0,\alpha}(\overline\Omega)}+(\omega_N\,R^N)^{-\frac{1}{\gamma}}\,\left(\int_{\Omega} |u(z)|^\gamma\,dz\right)^\frac{1}{\gamma}.
\]
We insert these estimates in \eqref{freddino} and use again the subadditivity of concave powers. We get
\[
\begin{split}
\frac{|u(x)-u(y)|}{|x-y|^\beta}\le 2\,[u]_{C^{0,\alpha}(\overline{\Omega})}^\frac{\beta}{\alpha}\,\left(R^{\alpha-\beta}\,[u]_{C^{0,\alpha}(\overline{\Omega})}^\frac{\alpha-\beta}{\alpha}+\left(\frac{1}{\omega_N\,R^N}\right)^\frac{\alpha-\beta}{\alpha\,\gamma}\,\|u\|^\frac{\alpha-\beta}{\alpha}_{L^\gamma(\Omega)}\right),
\end{split}
\]
which is valid for every $x\not=y$ and every $R>0$. If we not optimize in $R>0$, we finally get the desired estimate for the $C^{0,\beta}$ seminorm. The $\sup$ norm can be estimated with a similar optimization argument, by using \eqref{pazienza}.
\end{proof}
\begin{rem}
\label{rem:constantvaries}
We remark that the constants $C_1$ and $C_2$ in the previous result are given by
\[
C_1=2\,\omega_N^{-\frac{\alpha-\beta}{\gamma\,\alpha+N}}\,\left(\frac{\alpha\,\gamma}{N}\right)^\frac{N}{\alpha\,\gamma+N}\,\frac{1}{\chi},
\]
and
\[
C_2=\omega_N^{-\frac{\chi}{\gamma}}\, \left(\frac{\alpha\,\gamma}{N}\right)^\frac{N}{\alpha\,\gamma+N}\,\frac{1}{\chi}.
\]
In particular, we notice that $C_1$ stays uniformly bounded, as $\alpha$ varies in $(\beta,1]$. Similary, the constant $C_2$ stays bounded, as $\alpha$ varies in $(0,1]$. This observation will be useful in the sequel.
\end{rem}
	\subsection{Sobolev embeddings} 

	We recall the following family of {\it Gagliardo-Nirenberg interpolation inequalities}
	\begin{equation}
		\label{GNS}
		\|\psi\|_{L^r(\Omega)}\le G_{N,p,q,r}\,\|\nabla \psi\|_{L^p(\Omega)}^\vartheta\,\|\psi\|_{L^q(\Omega)}^{1-\vartheta},\qquad \mbox{ for every } \psi\in C^\infty_0(\Omega),
	\end{equation}
	which hold for every $1<p<\infty$, $1\le q\le p$ and $r$ satisfying 
	\[
	\begin{split}
		\begin{cases}
			q<r\le \dfrac{Np}{N-p}, \quad &\mbox{ if } 1 < p <N, \\
			q<r<\infty, &\mbox{ if } p=N,\\
			q<r\le \infty, \quad &\mbox{ if } p> N,
		\end{cases}
	\end{split}
	\]
	with $G_{N,p,q,r}>0$ only depending on $N,p,q, r$ and 
	\[
	\vartheta=\displaystyle \frac{N\,(r-q)}{N\,r\,(p-q)+p\,q\,r},
	\] 
	(for a proof, see for example \cite[Proposition 2.6]{BR}). 
	\vskip.2cm\noindent
	For $1<p<\infty$ and $\Omega\subset\mathbb{R}^N$ open set, we will indicate by $W^{1,p}_0(\Omega)$ the closure of $C^\infty_0(\Omega)$ in the usual Sobolev space $W^{1,p}(\Omega)$.	We will say that $\Omega$ is {\it $(p,q)-$admissible} if $\lambda_{p,q}(\Omega)>0$, the latter being defined in \eqref{lapq}.
	\par
	By using \eqref{GNS}, one can prove the following result:
	\begin{prop}\label{prop:spazi0}
		Let $1\le q<p<\infty$ and let  $\Omega\subseteq\mathbb{R}^N$ be an open set. Then
		\[ 
		\mathcal{D}^{1,p}_0(\Omega) \hookrightarrow L^q(\Omega) \qquad \Longleftrightarrow \qquad W^{1,p}_0(\Omega) \hookrightarrow L^q(\Omega).
		\]
		In particular, if $\Omega\subseteq\mathbb{R}^N$ is a $(p,q)-$admissible open set, then 
		\[
		W^{1,p}_0(\Omega)=\mathcal{D}^{1,p}_0(\Omega).
		\] 
	\end{prop}
	
	\begin{proof}
		The implication $\Longrightarrow$ is straightforward.	 For the converse implication, we suppose that there exists $C>0$ such that
		\[
		\|\psi\|_{L^q(\Omega)} \le C\, \|\psi\|_{W^{1,p}(\Omega)},\qquad \mbox{ for every } \psi\in C^\infty_0(\Omega).
		\]
By applying the Gagliardo-Nirenberg inequality \eqref{GNS} with $r=p$, it holds in particular
		\[ 
		\begin{split}
			\|\psi\|_{L^q(\Omega)} &\le C\, \|\psi\|_{W^{1,p}(\Omega)}\le C\, \left( \|\nabla \psi\|_{L^p(\Omega)} + G_{N,p,q,p} \, \|\nabla \psi\|^{\vartheta}_{L^p(\Omega)}  \|\psi\|^{1-\vartheta}_{L^q(\Omega)} \right).
		\end{split}
		\]
With a standard application of Young's inequality with exponents $1/\vartheta$ and $1/(1-\vartheta)$, we can absorb the $L^q$ norm on the right-hand side and get 
	\[ 
\|\psi\|_{L^q(\Omega)} \le \widetilde{C}\, \|\nabla \psi\|_{L^p(\Omega)}, 
	\]
	for some $\widetilde{C}>0$ independent of $\psi$. Thus, the embedding $\mathcal{D}^{1,p}_0(\Omega) \hookrightarrow L^q(\Omega)$ holds, as well.
	\par
		The last part of the statement follows by \cite[Proposition 2.5]{BPZ}.
	\end{proof}
	\begin{rem}
\label{oss:ugualip}
It is easily seen that the following implication holds
\[
\lambda_p(\Omega)>0 \qquad \Longrightarrow \qquad W^{1,p}_0(\Omega)=\mathcal{D}^{1,p}_0(\Omega).
\]	
Indeed, the positivity of $\lambda_p(\Omega)$ automatically gives that 
\[
\|\nabla u\|_{L^p(\Omega)}\qquad \mbox{ and }\qquad \|u\|_{W^{1,p}(\Omega)},
\]
are equivalent norms on $C^\infty_0(\Omega)$.
	\end{rem}
	The following simple result will be useful. Observe that we do not put any restriction either on $q$ or on the open set.
		\begin{lemma}
	\label{lm:inf}
	Let $1<p<\infty$, $1\le q\le \infty$ and let $\Omega \subsetneq \mathbb{R}^N$ be an open set. Then
	\[ 
	\lambda_{p,q}(\Omega)=\inf_{\psi \in C^{\infty}_0(\Omega)} \left\{\int_{\Omega} |\nabla \psi|^p \, dx\, :\, \|\psi\|_{L^q(\Omega)}=1\right\} = \inf_{\psi \in W^{1,p}_0(\Omega)} \left\{\int_{\Omega} |\nabla \psi|^p \, dx\, :\, \|\psi\|_{L^q(\Omega)} =1\right\}.
	\]
	\end{lemma}
\begin{proof}
We first observe that the problem on the right-hand side is actually settled on $W^{1,p}_0(\Omega)\cap L^q(\Omega)$. Since we have $C^\infty_0(\Omega)\subset W^{1,p}_0(\Omega)\cap L^q(\Omega)$, we immediately obtain
\[
	\inf_{\psi \in C^{\infty}_0(\Omega)} \left\{\int_{\Omega} |\nabla \psi|^p \, dx\, :\, \|\psi\|_{L^q(\Omega)} =1\right\} \ge \inf_{\psi \in W^{1,p}_0(\Omega)} \left\{\int_{\Omega} |\nabla \psi|^p \, dx\, :\, \|\psi\|_{L^q(\Omega)} =1\right\}.
\] 
In order to prove the reverse inequality, we first observe that if $\lambda_{p,q}(\Omega)=0$, then from the previous inequality we would get 
\[
0=\inf_{\psi \in C^{\infty}_0(\Omega)} \left\{\int_{\Omega} |\nabla \psi|^p \, dx\, :\, \|\psi\|_{L^q(\Omega)} =1\right\} \ge \inf_{\psi \in W^{1,p}_0(\Omega)} \left\{\int_{\Omega} |\nabla \psi|^p \, dx\, :\, \|\psi\|_{L^q(\Omega)} =1\right\}\ge 0,
\]
and thus the desired identity trivially holds true.
\par
Let us now suppose that $\lambda_{p,q}(\Omega)>0$. For every $u\in W^{1,p}_0(\Omega)\cap L^q(\Omega)$ not identically vanishing, there exists a sequence $\{u_n\}_{n\in\mathbb{N}}\subset C^\infty_0(\Omega)$ such that 
\[
\lim_{n\to\infty} \|\nabla u_n-\nabla u\|_{L^p(\Omega)}=\lim_{n\to\infty} \|u_n-u\|_{L^p(\Omega)}=0.
\]
By using the definition of $\lambda_{p,q}(\Omega)>0$, we have 
\[
\lambda_{p,q}(\Omega)\,\|u_n-u_m\|^p_{L^q(\Omega)}\le \|\nabla u_n-\nabla u_m\|^p_{L^p(\Omega)},\qquad \mbox{ for every } n,m\in\mathbb{N},
\] 
thus, in particular, $\{u_n\}_{n\in\mathbb{N}}$ is a Cauchy sequence in $L^q(\Omega)$. This shows that we have 
\[
\lim_{n\to\infty} \|u_n-u\|_{L^q(\Omega)}=0,
\]
as well.
Hence, we get
\[ 
\lambda_{p,q}(\Omega) \le \lim_{n\to\infty}\frac{\displaystyle\int_{\Omega} |\nabla u_n|^p \, dx}{\|u_n\|^p_{L^q(\Omega)}}=\frac{\displaystyle\int_{\Omega} |\nabla u|^p \, dx}{\|u\|_{L^q(\Omega)}^p}. 
\]
Finally, by taking the infimum on $W^{1,p}_0(\Omega) \cap L^q(\Omega)$ on the right-hand side, we obtain that
\[ 
\lambda_{p,q}(\Omega) \le\inf_{\psi \in W^{1,p}_0(\Omega)} \left\{\int_{\Omega} |\nabla \psi|^p \, dx\, :\, \|\psi\|_{L^q(\Omega)} =1\right\}.
\]
This concludes the proof.
\end{proof}
In the sequel, we will need the following Poincar\'e--type inequality, for functions which vanish  at a point of the  boundary. Obviously, this may hold only in the superconformal case $p>N$, i.e. in the case where points have positive $p-$capacity.
	\begin{prop}[A Poincar\'e inequality]\label{prop:poincare}
		Let $p>N$ and $u \in W^{1,p}(B_R(x_0))$ be such that $u(z)=0$, with $z \in \partial B_R(x_0)$. Then there exists a constant $C_{N,p}>0$ such that
		\begin{equation}\label{poincare}
			\int_{B_R(x_0)} |u|^p \, dx \le C_{N,p} \, R^p \, \int_{B_R(x_0)} |\nabla u|^p \, dx.
		\end{equation}
	\end{prop}
	\begin{proof}
		It is sufficient to consider the case $x_0=0$  and $R = 1$, then the general case follows by scaling and translating. For the proof of \eqref{poincare}, we can use a standard contradiction argument, exploiting compact Sobolev embeddings. We assume by contradiction that \eqref{poincare} fails to hold with a uniform constant. Thus, there exists a sequence $\{u_n\}_{n \in \mathbb{N}} \subset W^{1,p}(B_1)$ such that 
		\[ 
		\|u_n\|_{L^p(B_1)} = 1, \quad  \lim_{n \to \infty} \|\nabla u_n\|_{L^p(B_1)} = 0, \quad \mbox{ and } \quad u_n(z)=0 \mbox{ with } z \in \partial B_1. 
		\]
		In particular, $\{u_n\}_{n \in \mathbb{N}}$ is bounded in $W^{1,p}(B_1)$. Hence, thanks to compact Sobolev embeddings for $p>N$, there exists $u \in W^{1,p}(B_1)\cap C(\overline{B_1})$ such that $u_n$ converges to $u$  weakly in $W^{1,p}(B_1)$ and uniformly in $\overline{B_1}$, up to a subsequence. In particular, we get
		\[
		\|\nabla u\|_{L^p(B_1)}\le \liminf_{n\to\infty} \|\nabla u_n\|_{L^p(B_1)}=0 \qquad \mbox{ and }\qquad u(z)=\lim_{n\to\infty} u_n(z)=0.
		\]
		Since $B_1$ is a connected set, the last two facts imply that $u\equiv 0$ in $B_1$.	However, by the uniform convergence we also have 
		\[
		\|u\|_{L^{p}(B_1)}=\lim_{n\to\infty}\|u_n\|_{L^p(B_1)} = 1.
		\] 
		This contradicts the fact that $u$ identically vanishes.
	\end{proof}
	
	\subsection{The Lane-Emden equation}\label{sez:solutions}
	\begin{defn}
		Let $1 \le q< p < \infty$ and let $\Omega\subset\mathbb{R}^N$ be a $(p,q)-$admissible open set. We say that a function $v \in W^{1,p}_0(\Omega)$ is a {\it weak solution } of the Lane-Emden equation 
		\begin{equation}\label{eq:LE}
			-\Delta_{p} u = |u|^{q-2}\,u, \qquad \mbox{ in } \Omega,
		\end{equation} 
		if it satisfies 
		\begin{equation}\label{LEint}
			\int_{\Omega} \langle |\nabla v|^{p-2} \,\nabla v, \nabla \psi \rangle \, dx = \int_\Omega |v|^{q-2}\, v\, \psi \, dx, \quad \mbox{ for every } \psi \in C_0^{\infty}(\Omega). 
		\end{equation}
		In the case $q=1$, for a non-negative function $v$, we follow the convention 
		\[
		|v|^{q-2}\,v=v^{q-1}=1.
		\] 
By density, in the weak formulation \eqref{LEint} we can admit test functions in $W^{1,p}_0(\Omega)$.
	\end{defn}
	
	\begin{rem}[Scalings] 
		Let $1 \le q<p<\infty$. Given $t>0$, it is easily seen that if $u \in W^{1,p}_0(\Omega)$ is a weak solution of \eqref{eq:LE}, then the rescaled function
		\[
		u_t(x)=t^{\frac{p}{p-q}}\, u \left( \frac{x-x_0}{t} \right),
		\]
		is a weak solution of the same equation, in the new set $x_0+t\, \Omega$. On the other hand, if $u\in W^{1,p}_0(\Omega)$ weakly solves \eqref{eq:LE}, then the function 
		\[v=\alpha^\frac{1}{p-q}\,u\]
		is a weak solution of 
		\begin{equation}\label{eq:LEalpha}
			-\Delta_{p} u = \alpha\, |u|^{q-2}\,u, \qquad \mbox{ in } \Omega,
			\end{equation}
		with $\alpha>0$.
	\end{rem}
	We recall that if $\Omega\subseteq \mathbb{R}^N$ is a $(p,q)-$admissible {\it connected} open set for some $1\le q<p<\infty$, then there exists a unique weak positive solution of the following problem
		\begin{equation}\label{eq:cauchy}
			\begin{cases}
				-\Delta_p u = |u|^{q-2} u, \quad \mbox{ in } \Omega,\\
				u \in W^{1,p}_0(\Omega), \\
				u > 0, \quad \mbox{ in } \Omega,
			\end{cases}
		\end{equation}
		see \cite[Corollary 4.4]{BPZ}.
		We will denote this solution by $w_{p,q}^{\Omega}$.
		We recall that the latter also coincides with the unique positive solution of 
		\begin{equation}\label{eq:minprob}
			\min_{\psi\in W^{1,p}_0(\Omega)} \mathfrak{F}_{p,q}(\psi),
		\end{equation}
		where 
		\begin{equation}\label{eq:funzinteg}
			\mathfrak{F}_{p,q}(\psi):= \frac{1}{p} \int_\Omega |\nabla \psi|^{p} \, dx - \frac{1}{q} \int_\Omega |\psi|^{q} \, dx, \qquad \mbox{ for every } \psi\in W^{1,p}_0(\Omega),
		\end{equation}
		see  \cite[Theorem 4.3]{BPZ}.  
		By optimality and thanks to the relevant normalization condition on the $L^q$ norm, when $\Omega$ is a connected $(p,q)-$admissible open set, the positive minimizer of \eqref{lapq} coincides with the weak positive solution of \eqref{eq:LEalpha} corresponding to the choice $\alpha=\lambda_{p,q}(\Omega)$.
	\par
	The next proof is standard, we include it for completeness.
	\begin{prop}Let $1<p<\infty$, $1\le q<p$ and let $\Omega\subseteq\mathbb{R}^N$ be a $(p,q)$-admissible connected open set. Then 
		\begin{equation}\label{minprob}
			\min_{\psi\in W^{1,p}_0(\Omega)} \mathfrak{F}_{p,q}(\psi)=\frac{q-p}{p\,q}\,\left (  \frac{1}{\lambda_{p,q}(\Omega)}\right)^\frac{q}{p-q}
		\end{equation}
		and
		\begin{equation}\label{pqnorm}  
			\int_\Omega |\nabla w_{p,q}^{\Omega}|^{p} \, dx = \int_\Omega |w_{p,q}^{\Omega}|^{q} \, dx =  \left (\frac{1}{\lambda_{p,q}(\Omega)}\right)^\frac{q}{p-q}.
		\end{equation}
	\end{prop}
	
	\begin{proof} 
		Exploiting the different homogeneities of the two integrals and the fact that the minimum problem is equivalently settled on $W^{1,p}_0(\Omega)\setminus\{0\}$ (since a positive minimizer exists, see \cite[Theorem 3.3]{BPZ}), we get that 
		\[
		\min_{\psi\in W^{1,p}_0(\Omega)} \mathfrak{F}_{p,q}(\psi)= - \max_{\psi\in W^{1,p}_0(\Omega)\setminus\{0\}, t> 0} \left\{  \frac{t^q}{q} \int_\Omega |\psi|^{q} \, dx -\frac{t^p}{p} \int_\Omega |\nabla \psi|^{p} \, dx \right\}.
		\]
		It is easily seen that, for every $\psi\in W^{1,p}_0(\Omega)\setminus\{0\}$, the function
		\[t\mapsto\frac{t^q}{q} \int_\Omega  |\psi|^{q} \, dx- \frac{t^p}{p} \int_\Omega |\nabla \psi|^{p} \, dx\]
		is maximal  for 
		\[ 
		t_0= \left(\dfrac{\displaystyle\int_\Omega |\psi|^{q}\, dx}{\displaystyle\int_\Omega |\nabla \psi|^{p} \, dx}\right)^{\frac{1}{p-q}}.
		\]
		With such a choice of $t$, we get 
		\[ \frac{t_0^q}{q} \int_\Omega |\psi|^{q} \, dx-\frac{t_0^p}{p} \int_\Omega |\nabla \psi|^{p} \, dx  = \frac{p-q}{p\,q} \dfrac{\left(\displaystyle\int_\Omega |\psi|^{q}\, dx \right)^{\frac{p}{p-q}}}{\left(\displaystyle\int_\Omega  |\nabla \psi|^{p} \, dx \right)^{\frac{q}{p-q}}}. \]
		Then, by recalling the definition of $\lambda_{p,q}(\Omega)$, we get \eqref{minprob}.
		\par
		Finally, since $w_{p,q}^{\Omega}$ satifies the Lane-Emden equation \eqref{eq:LE}, by using this solution as a test function in \eqref{LEint}, we get that 
		\[
		\int_\Omega |\nabla w_{p,q}^{\Omega}|^{p} \, dx = \int_\Omega |w_{p,q}^{\Omega}|^{q} \, dx.
		\] 
		This implies
		\[\mathfrak{F}_{p,q}(w_{p,q}^{\Omega})=\frac{q-p}{p\,q }\int_\Omega |w_{p,q}^{\Omega}|^{q} \, dx.\] 
		Thanks to \eqref{minprob}, equality \eqref{pqnorm} easily follows.
	\end{proof}

	We now discuss the behavior of the $L^{\infty}$ norm of $w_{p,q}^{\Omega}$ when  $1\le q<\infty$ is fixed, $p$ goes to $\infty$ and $\Omega$ is a bounded convex open set.
	This result will be useful somewhere in Section \ref{Sec:asymp}.
	
	\begin{prop}\label{prop:convcase}
		Let $1<p<\infty$, $1\le q<p$ and $\Omega \subseteq\mathbb{R}^N$ be a
		bounded convex open set. Then 
		\[ \lim_{p \to \infty} \|w_{p,q}^{\Omega}\|_{L^{\infty}(\Omega)} = r_\Omega. \]
		In particular, it holds \begin{equation}\label{eq:behaviorpalla}  
		\lim_{p \to \infty} w_{p,q}^{B_1}(0) = 1.
		\end{equation}
	\end{prop}
	
	\begin{proof}
		For every $\eta \in C^{\infty}_0(\Omega)$ and $\varepsilon>0$, the function 
		\[
		\psi=\frac{|\eta|^p}{(w_{p,q}^{\Omega}+\varepsilon)^{p-1}} \in W^{1,p}_0(\Omega),
		\] 
		is a feasible test function in the equation \eqref{LEint} for $w_{p,q}^{\Omega}$. Hence, by applying Picone's inequality for the $p-$Laplacian (see \cite{AH}), we get that
		\[
		\begin{split}
			\int_\Omega (w_{p,q}^{\Omega})^{\,q-1}\, \frac{|\eta|^p}{(w_{p,q}^{\Omega}+\varepsilon)^{p-1}} \, dx
			&= \int_{\Omega} \left\langle |\nabla w_{p,q}^{\Omega}|^{p-2} \,\nabla w_{p,q}^{\Omega}, \nabla\left(\dfrac{|\eta|^p}{(w_{p,q}^{\Omega}+\varepsilon)^{p-1}}\right) \right\rangle \, dx \\
			&\le \int_\Omega |\nabla \eta|^p\,dx.
		\end{split}
		\] 
		Since $w_{p,q}^{\Omega} \in W^{1,p}_0(\Omega)$ is strictly positive by the minimum principle, by sending $\varepsilon$ to $0$, we get  
		\begin{equation}\label{hardylaneemden} 
			\int_\Omega \frac{|\eta|^p}{(w_{p,q}^{\Omega})^{\,p-q}} \, dx \le \int_\Omega |\nabla \eta|^p\,dx, \quad \mbox{ for every } \eta \in C^{\infty}_0(\Omega). 
		\end{equation} 
		By recalling the definition of $\lambda_p(\Omega)$ and that $w_{p,q}^\Omega$ is bounded, the latter estimate implies that
		\begin{equation}\label{eq:lap}
			1\le \| w_{p,q}^{\Omega} \|_{L^{\infty}(\Omega)}^{p-q} \, \lambda_{p}(\Omega).
		\end{equation}
		By raising to the power $1/p$ and using that
		\[
		\lim_{p\to\infty}\Big(\lambda_p(\Omega)\Big)^{\frac{1}{p}}= \frac{1}{r_\Omega},
		\] 
		(see \cite[Lemma 1.5]{JLM}), from \eqref{eq:lap} we obtain
		\[ 
		r_\Omega\le \liminf_{p \to \infty} \|w_{p,q}^{\Omega}\|_{L^{\infty}(\Omega)} .
		\]
		On the other hand, by using \cite[Corollary 5.3]{BPZ},
		we have that 
		\begin{equation}
		\label{boundACZ}
		\|w_{p,q}^{\Omega}\|_{L^{\infty}(\Omega)} \le \left(\frac{2}{\pi_{p,q}}\right)^\frac{p}{p-q}\,\left(\frac{q\,p-q+p}{p}\right)^\frac{1}{q} r_{\Omega}^{\frac{p-q} p},
		\end{equation}
		where $\pi_{p,q}$ is the following one-dimensional Sobolev-Poincar\'e constant
		\[
		\pi_{p,q}:=\Big(\lambda_{p,q}((0,1))\Big)^\frac{1}{p}.
		\]
		We claim that 
			\begin{equation}\label{A4}
			\pi_{p,q} \ge 2^{\left(1-\frac{1}{p}\right)} \left( q\left(1-\frac 1 p\right)+1\right)^{\frac{1}{q}},
		\end{equation}
		for every $1<p<\infty$ and $1\le q<\infty$. Postponing the proof of this fact for a moment, we see that \eqref{boundACZ} and \eqref{A4} would give
		\[
		\limsup_{p\to\infty}\|w_{p,q}^{\Omega}\|_{L^{\infty}(\Omega)} \le  r_{\Omega},
		\]
		as desired. Then the last part of the statement would follow by using that $w^{B_1}_{p,q}$ is radially symmetric decreasing, so that $w_{p,q}^{B_1}(0)= \|w_{p,q}^{B_1}\|_{L^{\infty}(B_1)}$.  
\par
We are left with establishing \eqref{A4}. Let $\varphi \in C^{\infty}_0((0,1))$, for every $t \in (0,1)$, we have that 
		\[
		|\varphi(t)|=|\varphi(t)-\varphi(0)| \le \int_{0}^{t} |\varphi'| \, dt\le t^{1-\frac{1}{p}}\, \| \varphi' \|_{L^p([0,1])},
		\]
		and 
		\[
		|\varphi(t)|=|\varphi(t)-\varphi(1)| \le \int_{t}^{1} |\varphi'| \, dt\le (1-t)^{1-\frac{1}{p}}\, \| \varphi' \|_{L^p([0,1])}.
		\]
		By raising to the power $q$ and integrating the first estimate on $(0,1/2)$ and the second one on $(1/2,1)$, we get
		\[
		\int_{0}^{1/2} |\varphi(t)|^q \, dt \le \frac{1}{q\left(1-\dfrac{1}{p}\right)+1}\,  \left(\frac{1}{2} \right)^{q\left(1-\frac{1}{p}\right)+1}  \, \| \varphi' \|_{L^{p}([0,1])}^q,
		\]
		and
		\[
		\int_{1/2}^{1} |\varphi(t)|^q \, dt \le \frac{1}{q\left(1-\dfrac{1}{p}\right)+1}\ \,\left(\frac{1}{2} \right)^{q\left(1-\frac{1}{p}\right)+1} \, \| \varphi' \|_{L^{p}([0,1])}^q.
		\]
		Then, by summing up, we obtain
		\[
		\int_{0}^{1} |\varphi(t)|^q \, dt \le \frac{1}{q\left(1-\dfrac{1}{p}\right)+1}  \left( \frac{1}{2}\right)^{q\left(1-\frac{1}{p}\right)} \, \| \varphi' \|_{L^{p}([0,1])}^q.
		\]
		By raising to the power $1/q$ on both sides and using the arbitrariness of $\varphi$, we get \eqref{A4}. 
	\end{proof}

	\begin{rem}  
		We point out that inequality \eqref{hardylaneemden} holds for general open sets: in this case, if the open set is not $(p,q)-$admissible, the function $w_{p,q}^{\Omega}$ has to be carefully defined. We refer to \cite[Proposition 4.5]{BR} for the case of the $p-$torsion function, i.e. the case  $q=1$ and $1<p<\infty$. The general case is contained in \cite[Lemma 4.1]{BFR} and \cite[Corollary 3.3]{TTD}. 
	\end{rem}

	\section{The distance function}\label{sez3}
	In this section we investigate some consequences of the summability of the distance function.
	First of all, we prove that when $d_{\Omega}^{\,\alpha}$ is summable for some $0<\alpha<\infty$, the set $\Omega$ has finite inradius. This comes with an explicit (sharp) bound.
	\begin{lemma}\label{lemma:r<infty}
		Let $0<\alpha<\infty$ and let $\Omega\subsetneq\mathbb{R}^N$ be an open set such that $d^\alpha_\Omega\in L^1(\Omega)$. Then $r_\Omega<+\infty$ and it holds
		\begin{equation}
			\label{inraggio}
			r_\Omega\le C_{N,\alpha}\,\left(\int_\Omega d_\Omega^\alpha\,dx\right)^\frac{1}{N+\alpha},
		\end{equation}
		where the constant $C_{N,\alpha}$ is given by
		\[
		C_{N,\alpha}=\left(N\,\omega_N\,\int_0^1 (1-\varrho)^\alpha\,\varrho^{N-1}\,d\varrho\right)^{-\frac{1}{N+\alpha}}.
		\]
		Moreover, inequality \eqref{inraggio} is sharp, since equality holds for a ball.
	\end{lemma}
	
	\begin{proof}
		Let $B_r(x_0)\subseteq\Omega$, then we have that
		\[
		(r-|x-x_0|)_+=d_{B_r(x_0)}(x)\le d_\Omega(x),\qquad \mbox{ for every } x\in B_r(x_0).
		\]
		By raising to the power $\alpha$ and integrating, we get 
		\[
		\int_{B_r(x_0)}(r-|x-x_0|)^\alpha_+\,dx\le \int_{\Omega} d_\Omega^\alpha\,dx.
		\]
		By using the change of variable $y=(x-x_0)/r$, from the previous estimate we also get
		\[
		r^{N+\alpha}\le \frac{\displaystyle\int_{\Omega} d_\Omega^\alpha\,dx}{\displaystyle\int_{B_1} (1-|y|)_+^\alpha\,dy}.
		\]
		If we now take the supremum over the admissible balls, we get the conclusion.
	\end{proof}
	\begin{lemma}
	\label{lm:distholder}
	Let $0<\alpha<\infty$ and let $\Omega\subsetneq\mathbb{R}^N$ be an open set such that $r_\Omega<+\infty$. Then for every $0<\beta<1$ we have 
	\[
	[d_\Omega]_{C^{0,\beta}(\overline\Omega)}\le (2\,r_\Omega)^{1-\beta}.
	\]
	\end{lemma}
	\begin{proof}
	We extend $d_\Omega$ to be $0$ outside $\Omega$ and consider it as a Lipschitz continuous function defined on the whole $\mathbb{R}^N$, By recalling \eqref{eq:normeHold}, for every $0<\beta<1$ we have
	\[
	[d_\Omega]_{C^{0,\beta}(\overline\Omega)}=[d_\Omega]_{C^{0,\beta}(\mathbb{R}^N)}.
	\]
It is now sufficient to write for $t>0$
\[
[d_\Omega]_{C^{0,\beta}(\mathbb{R}^N)}=\max\left\{\sup_{x\ne y;\, |x-y|\le t} \frac{|d_\Omega(x)-d_\Omega(y)|}{|x-y|^\beta},\, \sup_{\, |x-y|> t} \frac{|d_\Omega(x)-d_\Omega(y)|}{|x-y|^\beta}\right\}.
\]
For the first term on the right-hand side, we can just use the $1-$Lipschitz character of $d_\Omega$. For the second one, it is sufficient to use that $d_\Omega$ is bounded by $r_\Omega$. This gives
\[
[d_\Omega]_{C^{0,\beta}(\mathbb{R}^N)}=\max\left\{t^{1-\beta},\, \frac{2\,r_\Omega}{t^\beta}\right\},\qquad \mbox{ for every }t>0.
\]
By choosing $t=2\,r_\Omega$, we get the claimed estimate.
	\end{proof}
	\begin{lemma}
		\label{lm:controllo}
		Let $\Omega\subsetneq\mathbb{R}^N$ be an open set such that $d^\alpha_\Omega\in L^1(\mathbb{R}^N)$, for some $0<\alpha<\infty$. Then for every $R\ge r_\Omega/2$, we have
		\[
		d_\Omega(x)\le 2\,\left(\frac{1}{\omega_N}\,\int_{\mathbb{R}^N\setminus B_R} d_\Omega^\alpha\,dy\right)^\frac{1}{N+\alpha},\qquad \mbox{ for every }|x|> R+\frac{r_\Omega}{2}.
		\]
		In particular, $\Omega$ is quasibounded.
	\end{lemma}
	\begin{proof}
		Let $R\ge r_\Omega/2$ and let $x\in \mathbb{R}^N$ be such that $|x|>R+r_\Omega/2$. If $x\not\in\Omega$, then $d_\Omega(x)=0$ and there is nothing to prove. Let us suppose that $x\in\Omega$, so that $d_\Omega(x)>0$. We consider the ball
		\[
		B:=\Big\{y\in\mathbb{R}^N\, :\, |x-y|<\frac{d_\Omega(x)}{2}\Big\},
		\] 
		and observe that 
		\[
		d_\Omega(y)\ge \frac{1}{2}\,d_\Omega(x),\qquad \mbox{ for every } y\in B.
		\]
		We raise to the power $\alpha$ and integrate this inequality over $B$. This gives
		\[
		2^{-\alpha}\,d_\Omega(x)^\alpha\,|B|\le \int_{B} d_\Omega(y)^\alpha\,dy, 
		\]
		that is 
		\begin{equation}
			\label{dorme}
			\omega_N\,2^{-\alpha-N}\,d_\Omega(x)^{\alpha+N}\le \int_{B} d_\Omega(y)^\alpha\,dy.
		\end{equation}
		We then observe that 
		\[
		|y|\ge |x|-|y-x|> |x|-\frac{1}{2}\,d_\Omega(x)\ge |x|-\frac{r_\Omega}{2}> R,\qquad \mbox{ for every } y\in B.
		\]
		This gives that $B\subseteq\mathbb{R}^N\setminus B_{R}$.
		By using the previous inclusion in \eqref{dorme}, we get
		\[
		\omega_N\,2^{-\alpha-N}\,d_\Omega(x)^{\alpha+N}\le \int_{\mathbb{R}^N\setminus B_{R}} d_\Omega(y)^\alpha\,dy.
		\] 
		This concludes the proof.
	\end{proof}
		
	\begin{rem}
		The converse implication does not hold true: indeed, there exist quasibounded open sets for which $d^{\,\alpha}_\Omega\notin L^1(\mathbb{R}^N)$, for any $0<\alpha<\infty$ (see Example \ref{esempio}).
	\end{rem}
	
	We conclude this section with a generalized version of Ascoli-Arzel\`a Theorem, which is valid when $\Omega$ is a  quasibounded open set. This is probably well-known, but we have not been able to detect it in the literature. Thus, we include its proof.
	\begin{prop}\label{prop:AscArz} 
		Let $\Omega\subseteq \mathbb{R}^N$ be a quasibounded open set. 
		Let 
		\[
		\{u_n\}_{n\in\mathbb{N}}\subset \Big\{u\in C_{\rm bound}(\overline\Omega)\, :\, u=0 \mbox{ on }\partial\Omega\Big\}
				\] 
		be a sequence with the following properties:
		\begin{itemize}
		\item[(a)] there exists $M>0$ such that $\|u_n\|_{L^\infty(\Omega)}\le M$, for every $n\in\mathbb{N}$; 
		\vskip.2cm
		\item[(b)] there exist $\delta_0>0$ and a function $\omega:(0,\delta_0]\to (0,+\infty)$ such that
		\[
		\lim_{\delta\to 0^+} \omega(\delta)=0,
		\]
		and
		\[
		\omega_n(\delta):=\sup\Big\{|u_n(x)-u_n(y)|\, :\, x,y\in\overline\Omega,\ |x-y|\le \delta\Big\}\le \omega(\delta),\qquad \mbox{ for every } 0<\delta\le \delta_0.
		\]
		\end{itemize}  
		Then, there exists $\overline{u}\in C_{\rm bound}(\Omega)$ vanishing on $\partial\Omega$ such that
		\[
		\lim_{n\to\infty} \|u_n-\overline{u}\|_{L^\infty(\Omega)}=0,
		\]
		up to a subsequence. 
	\end{prop}
	
	\begin{proof} 
	Let us denote by $k_0\in\mathbb{N}$ the smallest natural number such that $\overline{\Omega}\cap \overline{B_k}$ is not empty. 
	 Thanks to the assumptions, $\{u_n\}_{n \in \mathbb{N}}$ is a bounded and equicontinuous sequence on the compact set $\overline{\Omega}\cap \overline{B_k}$, for every $k\ge k_0$. By applying the classical Ascoli-Arzel\'a Theorem for compact sets, together with a diagonal argument, we have that there exists a function $\overline{u} \in C(\overline{\Omega})$ such that, up to a subsequence, $u_n$ converges uniformly to $\overline{u}$ on $\overline{\Omega}\cap \overline{B_k}$, for every $k\ge k_0$. 
		\par
		We will show that $u_n$ converges to $\overline{u}$ uniformly on the whole $\overline\Omega$. 
		Let $0<\delta\le \delta_0$, since $\Omega$ is quasibounded, there exists $R_\delta>0$ such that
		\[
		d_{\Omega}(x) \le \delta, \qquad \mbox{ for } x\in \Omega\setminus B_{R_\delta}.
		\]
		For every $x \in \Omega \setminus B_{R_\delta}$, we take $y\in \partial{\Omega}$ such that $|x-y|=d_{\Omega}(x)$. By using property (b), the triangle inequality and the fact that $u_n(y)=0$, we get that for every $n\in\mathbb{N}$ 
		\[\begin{split}
			|u_n(x)-\overline{u}(x)| &= \lim_{m \to \infty} |u_n(x)-u_m(x)|\\
			&\le  |u_n(x)-u_n(y)|+\lim_{m\to\infty} |u_m(x)-u_m(y)|
			\le 2 \,\omega(d_{\Omega}(x))\le 2 \, \omega(\delta).
		\end{split}
		\]
		This implies that
		\[\begin{split} \|u_n - \overline{u}\|_{L^{\infty}(\Omega)} &=\max\Big\{ \|u_n - \overline{u}\|_{L^{\infty}(\Omega \setminus B_{R_\delta})}, \|u_n - \overline{u}\|_{L^{\infty}(\Omega\cap B_{R_\delta})}\Big\} \\
			&\le \max\Big\{2\, \omega(\delta), \|u_n - \overline{u}\|_{L^{\infty}(\Omega\cap B_{R_\delta})}\Big\}. 
		\end{split} \]
		By taking the limit as $n$ goes to $\infty$ and  exploiting the uniform convergence of $u_n$ to $\overline{u}$ on $\overline{\Omega}\cap \overline{B_{R_\delta}}$,  we obtain that
		\[
		\limsup_{n\to \infty} \|u_n - \overline{u}\|_{L^{\infty}(\Omega)} \le 2 \, \omega(\delta),\qquad \mbox{ for } 0<\delta\le \delta_0.
		\]
		By arbitrariness of $\delta$ and using the properties of $\omega$, we get
		\[
		\lim_{n\to \infty} \|u_n - \overline{u}\|_{L^{\infty}(\Omega)}=0.
		\]
		This also shows that $\overline{u}$ vanishes on $\partial\Omega$, as a uniform limit of functions with the same property. Moreover, using again that $C_{\rm bound}(\overline\Omega)$ is a Banach space, we finally get that $\overline{u}\in C_{\rm bound}(\overline\Omega)$, as well.
		This concludes the proof.
	\end{proof}

	\section{A Morrey--type inequality and its consequences}\label{sez4}
	
	We will need the following Morrey--type sharp constant.
	\begin{lemma}
		\label{lm:morreyA}
		Let $x_0\in\mathbb{R}^N$ and let $R>0$. For every $p>N$ and every fixed $z\in \partial B_R(x_0)$, we set 
		\[
		\mu_p(B_R(x_0);\{z\}):=\min_{\varphi\in W^{1,p}(B_R(x_0))}\left\{\int_{B_R(x_0)} |\nabla \varphi|^p\,dx\, :\, \varphi(x_0)=1 \mbox{ and } \varphi(z)=0 \right\}.
		\]
		Then:
		\begin{enumerate}
			\item the minimum above is independent of the point $z\in\partial B_R(x_0)$; 
			\vskip.2cm 
			\item we have the scaling 
			\[
			\mu_p(B_R(x_0);\{z\})=R^{N-p}\,\mu_p\left(B_1(x_0);\left\{\frac{z}{R}\right\}\right);
			\] 
			\item the family  
			\[ \left\{ \left(\frac{1}{\omega_N}\, \mu_p(B_R(x_0); \{z\}) \right)^{\frac{1}{p}} \right\}_{p>N}\] 
			is non-decreasing;
			\vskip.2cm
			\item we have the following asymptotics
			\begin{equation}\label{eq:limitemup}
				\lim_{p\to \infty} \Big(\mu_p(B_R(x_0);\{z\})\Big)^\frac{1}{p}=\frac{1}{R}.
			\end{equation}
		\end{enumerate} 
	\end{lemma}
	\begin{proof}
		We first show that $\mu_p\left(B_R(x_0);\{z\}\right)$ is actually a minimum. To this aim, we consider a minimizing sequence $\{u_n\}_{n \in \mathbb{N}} \subset W^{1,p}(B_R(x_0))$ for the problem defined by $\mu_p\left(B_R(x_0);\{z\}\right)$, i.\,e.
		\[ \lim_{n \to \infty} \int_{B_R(x_0)} |\nabla u_n|^p \, dx = \mu_p\left(B_R(x_0);\{z\}\right), \quad u_n(x_0)=1 \quad \mbox{ and } \quad u_n(z)=0. 
		\]
		By Proposition \ref{prop:poincare}, we also have that
		\[ 
		\int_{B_R(x_0)} |u_n|^p \, dx \le C_{N,p} \, R^p \, \int_{B_R(x_0)} |\nabla u_n|^p \, dx. 
		\]
		This implies that $\{u_n\}_{n \in \mathbb{N}}$ is a bounded sequence in $W^{1,p}(B_R(x_0)).$ Hence, thanks to compact Sobolev embeddings for $p>N$, there exists $u \in W^{1,p}(B_R(x_0))\cap C(\overline{B_R(x_0)})$ such that $u_n$ converges to $u$  weakly in $W^{1,p}(B_R(x_0))$ and uniformly in $\overline{B_R(x_0)}$, up to a subsequence.
		This implies that $u(x_0)=1$ and $u(z)=0$, thus $u$ is an admissible trial function for the problem $\mu_p\left(B_R(x_0);\{z\}\right)$. Moreover, by lower semicontinuity we have 
		\[ 
		\int_{B_R(x_0)} |\nabla u|^p \, dx \le \lim_{n \to \infty} \int_{B_R(x_0)} |\nabla u_n|^p \, dx = \mu_p\left(B_R(x_0);\{z\}\right), 
		\]
		i.\,e. $u$ is a minimizer.
		\par
		The proof of \noindent{\it part (1) and part (2)} easily follow by standard arguments, while \noindent{\it part (3)} is a consequence of H\"older's inequality. 
		\par
		It remains to prove \noindent{\it part (4)}. We suppose without loss of generality that $x_0=0$ and $R=1$. Thus, we have to show that 
		\[ 
		\lim_{p \to \infty} \Big(\mu_p\left(B_1;\left\{z\right\}\right)\Big)^{\frac{1}{p}} = 1.
		\]
		We observe that $d_{B_1}$ is admissible for the problem $\mu_p\left(B_1;\left\{z\right\}\right)$, thus we get that
		\[ 
		\Big(\mu_p\left(B_1;\left\{z\right\}\right)\Big)^{\frac{1}{p}} \le \omega_N^{\frac{1}{p}}. 
		\]
		This in turn implies that
		\begin{equation}\label{eq:boundmup} \limsup_{p \to \infty}\Big(\mu_p\left(B_1;\left\{z\right\}\right)\Big)^{\frac{1}{p}}\le \lim_{p \to \infty} \omega_N^{\frac{1}{p}} = 1. 
		\end{equation}
		We now prove that
		\[ \liminf_{p \to \infty}\Big(\mu_p\left(B_1;\left\{z\right\}\right)\Big)^{\frac{1}{p}}\ge  1. 
		\]
		At this aim, we consider $U_p \in W^{1,p}(B_1)$ a minimizer for $\mu_p\left(B_1;\left\{z\right\}\right)$, i.\,e.
		\[ \left(\int_{B_1} |\nabla U_p|^p \, dx\right)^{\frac{1}{p}} = \Big(\mu_p\left(B_1;\left\{z\right\}\right)\Big)^{\frac{1}{p}}, \qquad U_p(0)=1, \qquad \mbox{ and } \qquad U_p(z)=0. \]
		By applying Holder's inequality, we have
		\[ \left(\int_{B_1} |\nabla U_p|^{p_0} \, dx\right)^{\frac{1}{p_0}} \le \omega_N^{\frac{1}{p_0}-\frac{1}{p}} \left(\int_{B_1} |\nabla U_p|^p \, dx\right)^{\frac{1}{p}}, \]
		for every $p>p_0>N$. 
		Moreover, thanks again to Proposition \ref{prop:poincare}, we have
		\[ \int_{B_1} |U_p|^{p_0} \, dx \le C_{N,p} \, \int_{B_1} |\nabla U_p|^{p_0} \, dx.\]
		Taking into account \eqref{eq:boundmup}, the above inequality 
		implies  that the sequence $\{U_p\}_{p>N}$ is bounded in $W^{1,p_0}(B_1)$ for every fixed $p_0>N$. Then, there exists $U_{\infty} \in W^{1,p_0}(B_1)\cap C(\overline{B_1})$ such that $\{U_p\}_{p>N}$ converges to $U_{\infty} $  weakly in $W^{1,p_0}(B_1)$ and uniformly in $\overline{B_1}$, up to taking a sequence.
		Thanks to a standard argument, we have that $\{U_p\}_{p>N}$ converges to $U_{\infty}$ weakly in $W^{1,q}(B_1)$ for every $p_0\le q<\infty$. Thanks to the claimed convergence, we get that
		\[ 
		\begin{split}
			\left(\int_{B_1} |\nabla U_\infty|^{q} \, dx\right)^{\frac{1}{q}} &\le \liminf_{p \to \infty} \left(\int_{B_1} |\nabla U_p|^{q} \, dx\right)^{\frac{1}{q}} \\ 
			&\le \liminf_{p \to \infty} \, \omega_N^{\frac{1}{q}-\frac{1}{p}} \left(\int_{B_1} |\nabla U_p|^p \, dx\right)^{\frac{1}{p}} \\
			&= \omega_N^{\,\frac{1}{q}}\, \liminf_{p \to \infty} \Big(\mu_p(B_1; \{z\})\Big)^{\frac{1}{p}}, \qquad \mbox{ for every } q \ge p_0,
		\end{split}
		\]
		and, by sending $q$ to $\infty$ and recalling \eqref{eq:boundmup}, it holds that
		\begin{equation}\label{eq:liminfmup}  1\ge \liminf_{p \to \infty} \Big(\mu_p(B_1; \{z\})\Big)^{\frac{1}{p}} \ge \|\nabla U_\infty\|_{L^{\infty}(B_1)} .
		\end{equation}
		Hence $U_{\infty}$ is a $1-$Lipschitz continuous function. Accordingly, we get
		\[\|\nabla U_\infty\|_{L^{\infty}(B_1)} \ge  \frac{|U_{\infty}(0)|}{d_{B_1}(0)}=1,\]
		which, combined with \eqref{eq:liminfmup}, gives the conclusion.
	\end{proof}
	
	\begin{rem}
		\label{rem:morreytype}
		Thanks to Lemma \ref{lm:morreyA} {\it part (1)}, fixed a point $z\in\partial B_1$, we can  define \begin{equation}
		\label{def:mupB1}	
		\mu_p(B_1):=\min_{u\in W^{1,p}(B_1)}\left\{\int_{B_1} |\nabla u|^p\,dx\, :\, u(0)=1 \mbox{ and } u(z)=0 \right\}.
		\end{equation}
		Moreover, as an easy consequence of Lemma \ref{lm:morreyA}, when $p>N$ we get the following estimate 
		\begin{equation}\label{eq:morreylemma4.1}
			|u(x_0)-u(z)| \le \frac{R^{1-\frac{N}{p}}}{\Big(\mu_p(B_1)\Big)^{\frac{1}{p}}}\, \|\nabla u\|_{L^p(B_R(x_0))},\quad \mbox{ for } u\in W^{1,p}(B_R(x_0)) \mbox{ and  } z\in\partial B_R(x_0).
		\end{equation}
	\end{rem}
	
		\begin{cor}[Sharp Morrey constant]
		\label{cor:morrey} Let $N<p<\infty$ and let $\Omega\subseteq \mathbb{R}^N$ be an open set. We define the sharp Morrey constant
		\[
		\mathfrak{m}_{p}(\Omega):=\inf_{u\in C_0^{\infty}(\Omega)} \left\{\int_\Omega |\nabla u|^p\,dx\, :\, [u]_{C^{0,\alpha_p}(\overline\Omega)}=1 \right\},\qquad \mbox{ where }\alpha_p:=1-\frac{N}{p}.
		\]
		Then the constant $\mathfrak{m}_p(\Omega)$ is independent of $\Omega$, i.e. we have
		\[
		\mathfrak{m}_p(\Omega)=\mathfrak{m}_p(\mathbb{R}^N).
		\]
		Moreover, we have
		\begin{equation}\label{eq:mp}
			\mu_p(B_1)\le \mathfrak{m}_{p}(\mathbb{R}^N)\le N\,\omega_N\,\left(\frac{p-N}{p-1}\right)^{p-1},
		\end{equation}
		and 
		\[ 
		\lim_{p\to \infty} \Big(\mathfrak{m}_{p}(\mathbb{R}^N)\Big)^\frac{1}{p}=1.
		\]
	\end{cor}
	\begin{proof} 
		We first show that $\mathfrak{m}_p(\Omega)$ is independent of $\Omega$. The fact that $\mathfrak{m}_p(\mathbb{R}^N) \le \mathfrak{m}_p(\Omega)$ follows by monotonicity with respect to sets inclusion and the fact that 
		\[
			[u]_{C^{0,\alpha_p}(\overline\Omega)}= [u]_{C^{0,\alpha_p}(\mathbb{R}^N)},\qquad \mbox{ for every } u\in C^\infty_0(\Omega),
	\]
	see \eqref{eq:normeHold}.\par
		In order to show that $\mathfrak{m}_p(\Omega) \le \mathfrak{m}_p(\mathbb{R}^N)$, let $u \in C^{\infty}_0(\mathbb{R}^N)$ and let 
		\[
		u_r(x)=u\left(\frac{x-x_0}{r}\right),\qquad \mbox{ with } x_0\in \mathbb{R}^N \mbox{ and } r>0.
		\] 
		Since $u$ has compact support, we have $u_r \in C^{\infty}_0(\Omega)$ for some suitable $x_0$ and $r$ small enough. Then, by scaling and thanks to \eqref{eq:normeHold}, it holds
		\[ 
		\mathfrak{m}_p(\Omega) \le \dfrac{\displaystyle\int_{\Omega} |\nabla u_r|^p \, dx}{[u_r]^p_{C^{0,\alpha_p}(\overline\Omega)}} = \dfrac{\displaystyle\int_{\mathbb{R}^N} |\nabla u_r|^p \, dx }{[u_r]^p_{C^{0,\alpha_p}(\mathbb{R}^N)}} = \dfrac{\displaystyle\int_{\mathbb{R}^N} |\nabla u|^p \, dx }{[u]^p_{C^{0,\alpha_p}(\mathbb{R}^N)}}. 
		\]
		By taking the infimum on $C^{\infty}_0(\mathbb{R}^N)$ on the right-hand side, we get the claimed inequality. 
		\vskip.2cm\noindent
		We now come to the proof of \eqref{eq:mp}. Let $u\in C_0^{\infty}(\mathbb{R}^N)$, for the lower bound it is sufficient to prove that
		\begin{equation}\label{eq:morrey1*}
			|u(x)-u(y)|\le \frac 1 {\Big(\mu_p(B_1)\Big)^{\frac{1}{p}}}\,\|\nabla u\|_{L^p(\Omega)}\, |x-y|^{\alpha_p}, \qquad \mbox{ for every } x,y\in\mathbb{R}^N. 
		\end{equation}
		If $x=y$, then \eqref{eq:morrey1*} trivially holds, thus let us assume that $x\not =y$. Without loss of generality, we assume $u(x)>u(y)$ and we define 
		\[
		v(z):=\dfrac{u(z)-u(y)}{u(x)-u(y)},\qquad \mbox{ for } z\in\mathbb{R}^N,
		\] 
		which satisfies $v(x)=1$ and $v(y)=0$. Since $v\in W^{1,p}(B_R(x))$ with $R=|x-y|$, we have from \eqref{eq:morreylemma4.1} that 
		\[
		1 = |v(x)| \le \frac{|x-y|^{\alpha_p}}{\Big(\mu_p(B_1)\Big)^{\frac{1}{p}}}\, \|\nabla v\|_{L^p(B_R(x))} = \frac{|x-y|^{\alpha_p}}{\Big(\mu_p(B_1)\Big)^{\frac{1}{p}}} \, \frac{1}{|u(x)-u(y)|} \, \|\nabla u\|_{L^p(B_R(x))}.
		\]
		From this estimate, we get  
		\[
		|u(x)-u(y)|\le \frac{|x-y|^{\alpha_p}}{\Big(\mu_p(B_1)\Big)^{\frac{1}{p}}}\, \|\nabla u\|_{L^p(B_R(x))}\le \frac{|x-y|^{\alpha_p}}{\Big(\mu_p(B_1)\Big)^{\frac{1}{p}}} \,\|\nabla u\|_{L^p(\mathbb{R}^N)}, 
		\]
		which is the claimed inequality \eqref{eq:morrey1*}.
		\par
		As for the upper bound, for every $u\in C^\infty_0(B_1)\setminus\{0\}$, by the first part of the proof and the very definition of $\mathfrak{m}_p$, we have
		\[
		\mathfrak{m}_p(\mathbb{R}^N)=\mathfrak{m}_p(B_1)\le \frac{\|\nabla u\|_{L^p(B_1)}^p}{[u]^p_{C^{0,\alpha_p}(B_1)}}.
		\]
		Moreover, by using that $u$ is compactly supported in $B_1$, we have
		\[
		[u]_{C^{0,\alpha_p}(B_1)}\ge \sup_{x\in B_1, y\in \partial B_1} \frac{|u(x)|}{|x-y|^{\alpha_p}}\ge |u(0)|.
		\]
		Thus, for every $u\in C^\infty_0(B_1)$ such that $|u(0)|\not =0$, we get 
		\[
		\mathfrak{m}_p(\mathbb{R}^N)\le \frac{\|\nabla u\|^p_{L^p(B_1)}}{|u(0)|^p}.
		\]
		By density, the last estimate is still true for functions $u\in W^{1,p}_0(B_1)$. Now we consider the function 
		\[ u(x)=(1-|x|^{\frac{p-N}{p-1}}) \in W^{1,p}_0(B_1),\] 
		hence there exists $u_n\in C_0^{\infty}(B_1)$ such that $u_n$ converges to $u$ in $W^{1,p}_0(B_1)$.
		Since
		\[ 
		[u]_{C^{0,\alpha_p}(\overline{B_1})} = \sup_{x\ne y, x,y\in \overline{B_1} } \frac{\left||x|^{\frac{p-N}{p-1}}-|y|^{\frac{p-N}{p-1}}\right|}{|x-y|^{\alpha_p}} \ge 1, 
		\]
		it holds that
		\[ 
		\begin{split}
			\Big(\mathfrak{m}_{p}(\mathbb{R}^N)\Big)^{\frac{1}{p}} \le \liminf_{n\to \infty} \frac{\|\nabla u_n\|_{L^p(\mathbb{R}^N)}}{ [u_n]_{C^{0,\alpha_p}(\mathbb{R}^N)}}& \le \limsup_{n\to \infty} \frac{(N \,\omega_N)^\frac{1}{p} \, \left( \dfrac{p-N}{p-1} \right)^{\frac{p-1}{p}}}{[u_n]_{C^{0,\alpha_p}(\overline{B_1})}} \le (N \,\omega_N)^\frac{1}{p} \,\left( \dfrac{p-N}{p-1} \right)^{\frac{p-1}{p}}. 
		\end{split}
		\]
		This shows the claimed upper bound.
		\vskip.2cm\noindent
		Finally, by taking the $p-$rooth in \eqref{eq:mp} and using \eqref{eq:limitemup}, we get the desired asymptotics for $\mathfrak{m}_p$.
	\end{proof}
	The following Hardy inequality for general open sets was originally proved for $q=p$ in \cite{Lewis} (see also \cite{Ha} and \cite{Wa}), without determination of an explicit constant. The latter can be found in \cite{Av, Ch}.
	We generalize the result to cover the case $p\le q\le \infty$. We will pay due attention to the asymptotic behaviour of the sharp constant, as $p$ goes to $\infty$.
\begin{thm}[Hardy's inequality] 
\label{teo:hardyineq}
	Let $N < p\le q \le \infty$ and let $\Omega\subsetneq \mathbb{R}^N$  be an open set. We set
	\[
	\mathfrak{h}_{p,q}(\Omega)=\inf_{u\in C^\infty_0(\Omega)}\left\{\int_\Omega |\nabla u|^p\,dx\, :\, \left\|\frac{u}{d_\Omega^{\frac{N}{q}+\frac{p-N}{p}}}\right\|_{L^q(\Omega)}=1\right\},\qquad \mbox{ for } p<q\le \infty,
	\]
and
	\[
	\mathfrak{h}_{p}(\Omega)=\inf_{u\in C^\infty_0(\Omega)}\left\{\int_\Omega |\nabla u|^p\,dx\, :\, \left\|\frac{u}{d_\Omega}\right\|_{L^p(\Omega)}=1\right\}.
	\]
We have that
		\begin{equation}\label{lowerboundhardy}
\mathfrak{h}_{p,q}(\Omega) \ge \Big(\mathfrak{h}_p(\Omega)\Big)^\frac{p}{q}\,\Big(\mathfrak{h}_{p,\infty}(\Omega)\Big)^\frac{q-p}{q},\qquad \mbox{ for } p<q<\infty,
\end{equation}
and
	\begin{equation}\label{lowerboundhardyext}
	\mathfrak{h}_p(\Omega)\ge \left(\frac{p-N}{p}\right)^p,\qquad \mathfrak{h}_{p,\infty}(\Omega)\ge \mu_p(B_1),
	\end{equation}	
	where $\mu_p(B_1)$ is the same constant as in \eqref{def:mupB1}.
Moreover, it holds
\begin{equation}
		\label{hardyinfty}
		\lim_{p\to \infty} \Big(\mathfrak{h}_{p}(\Omega)\Big)^{\frac{1}{p}}=\lim_{p\to \infty} \Big(\mathfrak{h}_{p,\infty}(\Omega)\Big)^{\frac{1}{p}}=1.
	\end{equation}
\end{thm} 
\begin{proof}
We first prove the lower bound in the extremal case, i.e. for $q=\infty$.
Let $x\in \Omega$ and let $\overline{x}\in\partial\Omega$ be such that 
\[
|x-\overline{x}|=d_\Omega(x).
\]
For every $u\in C^\infty_0(\Omega)$ and every $p>N$, we thus get
\begin{equation}
\label{base}
|u(x)|\le \frac{d_\Omega(x)^{1-\frac{N}{p}}}{(\mu_p(B_1))^\frac{1}{p}}\,\left(\int_{B_{d_\Omega(x)}(x)} |\nabla u|^p\,dx\right)^\frac{1}{p}.
\end{equation}
By taking the supremum over $\Omega$, we get
\[
\left\|\frac{u}{d_\Omega^{1-\frac{N}{p}}}\right\|_{L^\infty(\Omega)}\le \frac{1}{(\mu_p(B_1))^\frac{1}{p}}\,\left(\int_{\mathbb{R}^N} |\nabla u|^p\,dx\right)^\frac{1}{p}, \qquad \mbox{ for every } u\in C^\infty_0(\Omega).
\]
This gives the desired Hardy inequality result for $q=\infty$, together with the claimed lower bound in \eqref{lowerboundhardyext}.
In the case $p=q$, the estimate in \eqref{lowerboundhardyext} comes from \cite{Av, Ch}, as already recalled.
\par
The case $p<q<\infty$ now simply follows from interpolation of the two endpoints. Indeed, for every $u\in C^\infty_0(\Omega)$, we have 
\[
\left(\int_\Omega \frac{|u|^q}{d_\Omega^{\gamma\,q}}\,dx\right)^\frac{p}{q} \le \left(\int_\Omega \frac{|u|^p}{d_\Omega^p}\,dx\right)^\frac{p}{q}\,\left\|\frac{u}{d_\Omega^\frac{\gamma\,q-p}{q-p}}\right\|_{L^\infty(\Omega)}^{(q-p)\,\frac{p}{q}}
\]
where we set for simplicity
\[
\gamma=\frac{N}{q}+\frac{p-N}{p}.
\]
We observe that 
\[
\frac{\gamma\,q-p}{q-p}=1-\frac{N}{p},
\]
thus by using the definitions of $\mathfrak{h}_p(\Omega)$ and $\mathfrak{h}_{p,\infty}(\Omega)$, we obtain 
\[
\left(\int_\Omega \frac{|u|^q}{d_\Omega^{\gamma\,q}}\,dx\right)^\frac{p}{q}\le \left(\frac{1}{\mathfrak{h}_p(\Omega)}\right)^\frac{p}{q}\,\left(\frac{1}{\mathfrak{h}_{p,\infty}(\Omega)}\right)^\frac{q-p}{q}\,\int_\Omega |\nabla u|^p\,dx.
\]
By taking the infimum over $u\in C^\infty_0(\Omega)$, we get the lower bound \eqref{lowerboundhardy}.
\vskip.2cm\noindent
In order to prove the last statement, for the case $q=\infty$, from \eqref{lowerboundhardyext} and Lemma \ref{lm:morreyA} we have 
\[
\liminf_{p\to\infty} \Big(\mathfrak{h}_{p,\infty}(\Omega)\Big)^\frac{1}{p}\ge \lim_{p\to\infty}\Big(\mu_p(B_1)\Big)^\frac{1}{p}=1.
\]
In the case $p=q$, we directly have 
\[
\liminf_{p\to\infty} \Big(\mathfrak{h}_{p}(\Omega)\Big)^\frac{1}{p}\ge \lim_{p\to\infty} \frac{p-N}{p}=1.
\]
In order to prove that the $\limsup$ is smaller than or equal to $1$, it is sufficient to use a suitable trial function: for every $x_0\in \Omega$, we have that
\[
\varphi(x)=\big(r-|x-x_0|\big)_+\in W^{1,p}_0(\Omega),\qquad \mbox{ for } r=d_\Omega(x_0).
\]
Since $p>N$, we can infer existence of a sequence $\{\varphi_n\}_{n\in\mathbb{N}}\subset C^\infty_0(\Omega)$ such that 
\[
\lim_{n\to\infty} \|\nabla \varphi_n-\nabla \varphi\|_{L^p(\Omega)}=\lim_{n\to\infty} \|\varphi_n-\varphi\|_{L^\infty(\Omega)}=0.
\]
Thus we get 
\[
\Big(\mathfrak{h}_{p,\infty}(\Omega)\Big)^\frac{1}{p}\le \lim_{n\to\infty}\frac{\|\nabla \varphi_n\|_{L^p(\Omega)}}{\displaystyle\left\|\frac{\varphi_n}{d_\Omega^{\frac{p-N}{p}}}\right\|_{L^\infty(\Omega)}}=\frac{(\omega_N\,r^N)^\frac{1}{p}}{\displaystyle\left\|\frac{\varphi}{d_\Omega^{\frac{p-N}{p}}}\right\|_{L^\infty(\Omega)}},
\]
and
\[
\Big(\mathfrak{h}_{p}(\Omega)\Big)^\frac{1}{p}\le \lim_{n\to\infty}\frac{\|\nabla \varphi_n\|_{L^p(\Omega)}}{\displaystyle\left\|\frac{\varphi_n}{d_\Omega}\right\|_{L^p(\Omega)}}=\frac{(\omega_N\,r^N)^\frac{1}{p}}{\displaystyle\left\|\frac{\varphi}{d_\Omega}\right\|_{L^p(\Omega)}}.
\]
By using that 
\[
\lim_{p\to\infty} \frac{(\omega_N\,r^N)^\frac{1}{p}}{\displaystyle\left\|\frac{\varphi}{d_\Omega^{\frac{p-N}{p}}}\right\|_{L^\infty(\Omega)}}=\lim_{p\to\infty} \frac{(\omega_N\,r^N)^\frac{1}{p}}{\displaystyle\left\|\frac{\varphi}{d_\Omega}\right\|_{L^p(\Omega)}}=\inf_{x\in B_r(x_0)} \frac{d_\Omega(x)}{(r-|x-x_0|)_+}\le \frac{d_\Omega(x_0)}{r}=1,
\]
we then obtain the desired conclusion.
\end{proof}
 \begin{rem}\label{rem:estensione} 
	By a standard density argument, for every $p>N$ and $p\le q\le \infty$ the Hardy inequality 
	\[
	\mathfrak{h}_{p,q}(\Omega)\,\left\|\frac{u}{d_\Omega^{\frac{N}{q}+\frac{p-N}{p}}}\right\|_{L^q(\Omega)}^p\le \int_\Omega |\nabla u|^p\,dx,
	\]
	still holds in both spaces $\mathcal{D}^{1,p}_0(\Omega)$ and $W^{1,p}_0(\Omega)$, for every open set $\Omega\subsetneq\mathbb{R}^N$.
\end{rem}

	\section{Embedding theorems}\label{sez5}
	
	For the ease of presentation of our main embedding results, we distinguish between three cases: 
	\begin{itemize}
		\item the case $q<p$; 
		\vskip.2cm
		\item the case $q=p$ with $\Omega$ having finite inradius;
		\vskip.2cm
		\item the case $q=p$ with $\Omega$ being a quasibounded set.
	\end{itemize}
	\vskip.2cm
Then, in the final subsection, we will briefly discuss the case $q>p$.	
	\subsection{The case $q<p$}
	
	\begin{thm}\label{teo:q<p} 
		Let $1\le q<p<\infty$ and let $\Omega\subsetneq \mathbb{R}^N$ be an open set. The following facts hold:
		\begin{enumerate}
			\item[(i)] we have that  
			\[
			\mathcal D^{1,p}_0(\Omega)\hookrightarrow L^q(\Omega) \qquad \Longrightarrow \qquad d_{\Omega} \in L^{\frac{p\,q}{p-q}}(\Omega),
			\] 
			and the following upper bound holds 
			\begin{equation}\label{eq:pqstima1}
				\lambda_{p,q}(\Omega) \,\left(\displaystyle\int_{\Omega} d_{\Omega}^{\frac{p\,q}{p-q}} \, dx \right)^{\frac{p-q}{q}} \le \lambda_p(B_1);
			\end{equation}
			\vskip.2cm
			\item[(ii)] moreover, if $N<p<\infty$, then we also have
			\[  
			d_{\Omega} \in L^{\frac{p\,q}{p-q}}(\Omega)\qquad \Longrightarrow\qquad \mathcal D^{1,p}_0(\Omega)\hookrightarrow L^q(\Omega),
			\]
			and the following lower bound holds			
			\begin{equation}\label{eq:pqstima2}
				\mathfrak{h}_p(\Omega) \le \lambda_{p,q}(\Omega) \,\left(\displaystyle\int_{\Omega} d_{\Omega}^{\frac{p\,q}{p-q}} \, dx \right)^{\frac{p-q}{q}},
			\end{equation}
			where $\mathfrak{h}_p(\Omega)$ is the sharp Hardy constant (see Theorem \ref{teo:hardyineq});
			\vskip.2cm
			\item[(iii)] finally, if $p\le N$, there exists an open set $\mathcal{T}\subsetneq \mathbb{R}^N$ such that 
			\[
			d_{\mathcal{T}}\in L^1(\mathcal{T})\cap L^{\infty}(\mathcal{T})\qquad  \mbox{ but }\qquad \mathcal{D}^{1,p}_0(\mathcal{T})\not\hookrightarrow L^q(\mathcal{T}).
			\]
		\end{enumerate}
	\end{thm}
	
	\begin{proof}
		We prove each point separately.
		\begin{enumerate}
			\item[(i)] Let $x_0\in \Omega$. Since both $w_{p,q}^{B_1}$ and $w_{p,q}^{\Omega}$ are continuous functions, evaluating the lower bound in \cite[Theorem 5.2]{BPZ} at $x=x_0$ and $r=d_{\Omega}(x_0)$, we get 
			\begin{equation}
				\label{eq:lowest}
				d_{\Omega}(x_0)^{\frac{p}{p-q}}\,w_{p,q}^{B_1}(0) \le  w_{p,q}^{\Omega} (x_0). 
			\end{equation}
			Then, by raising to the power $q$ both sides of  \eqref{eq:lowest}, integrating on $\Omega$ and exploiting \eqref{pqnorm}, we get
			\[
			\int_{\Omega} d_{\Omega}^{\frac{p\,q}{p-q}} \, dx \le \big(w_{p,q}^{B_1}(0)\big)^{-q} \, \int_{\Omega} (w_{p,q}^{\Omega})^q(x) \, dx= \big(w_{p,q}^{B_1}(0)\big)^{-q} \,\left(\frac{1}{\lambda_{p,q}(\Omega)}\right)^{\frac{q}{p-q}}.
			\]
			By using \eqref{eq:lap} for the ball $B_1$
			\[ 
			\big(w_{p,q}^{B_1}(0)\big)^{-q} \le \Big(\lambda_p(B_1)\Big)^{\frac{q}{p-q}}, 
			\]
			we get the claimed summability of $d_\Omega$, together with the upper bound in \eqref{eq:pqstima1}. 
			\vskip.2cm
			\item[(ii)] Let us suppose that $d_{\Omega} \in L^{\frac{p\,q}{p-q}}(\Omega)$ and $p>N$. For every $u\in C^{\infty}_0(\Omega)$, a joint application of H\"older's and Hardy's inequalities (see Theorem \ref{teo:hardyineq}) leads to
			\[
			\begin{split}
				\int_{\Omega} |u|^q \, dx &\le \left( \int_{\Omega} \frac{|u|^p}{d_{\Omega}^{\,p}} \, dx \right)^\frac{q}{p}\, \left( \int_{\Omega} d_{\Omega}^{\frac{p\,q}{p-q}} \, dx\right)^{\frac{p-q}{p}} \\
				&\le \Big(\mathfrak{h}_p(\Omega)\Big)^{-\frac{q}{p}} \left( \int_{\Omega} |\nabla u|^p \, dx \right)^{\frac{q}{p}} \left( \int_{\Omega} d_{\Omega}^{\frac{p\,q}{p-q}} \, dx\right)^{\frac{p-q}{p}}.
			\end{split}
			\]
			This in turn implies that
			\[
			\dfrac{\displaystyle\int_{\Omega} |\nabla u|^p \, dx}{\left( \displaystyle\int_{\Omega} |u|^q \, dx \right)^{\frac{p}{q}}}\ge \dfrac{\mathfrak{h}_p(\Omega)}{\left( \displaystyle\int_{\Omega} d_{\Omega}^{\frac{p\,q}{p-q}} \, dx\right)^{\frac{p-q}{q}}}.
			\]
			By taking the infimum on $C^{\infty}_0(\Omega)$ on the left-hand side, we get the lower bound in \eqref{eq:pqstima2}. This in particular shows that $\lambda_{p,q}(\Omega)>0$, i.e. we have the embedding
			\[
			\mathcal{D}^{1,p}_0(\Omega)\hookrightarrow L^q(\Omega),
			\] 
			as desired.
			\vskip.2cm
			\item[(iii)] We construct an open set $\mathcal{T} \subseteq\mathbb{R}^N$ such that, under the assumption $1 < p \le N$
			\vskip.2cm
			\begin{itemize}
				\item $d_{\mathcal{T}} \in L^1(\mathcal{T}) \cap L^\infty(\mathcal{T})$, hence $d_{\mathcal{T}} \in L^\alpha(\mathcal{T})$  for every $\alpha \in [1, +\infty]$;
				\vskip.2cm
				\item $\mathcal{D}^{1,p}_0(\mathcal{T})$ is not compactly embedded in $L^q(\mathcal{T})$, for every $1 \le q < p$.
			\end{itemize}
			\vskip.2cm
			\begin{figure}
				\includegraphics[scale=.4]{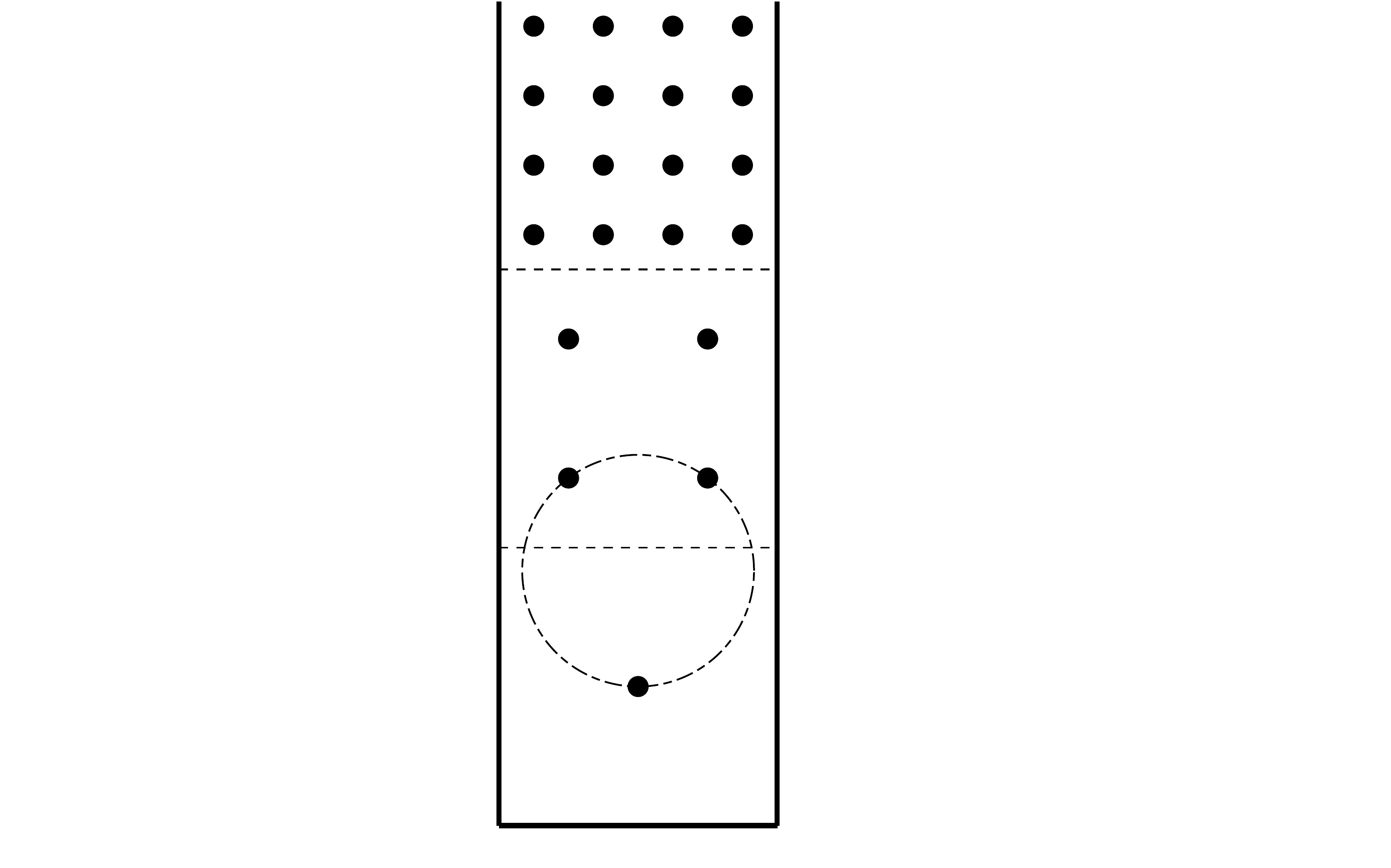}
				\caption{The construction of the set $\mathcal{T}$. The horizontal dashed lines denote the separation lines between the cubes $C_k$. The dashed circular line highlights the ball with maximal radius $r_\mathcal{T}$.}
			\end{figure}
			We consider the $(N-1)-$dimensional open hypercube $Q=(0,1)^{N-1} \subseteq\mathbb{R}^{N-1}$ and we define
			\[
			C_k=Q \times (k, k+1],\qquad \mbox{ for every } k \in \mathbb{N}.
			\]	
			Then, for every $k \in \mathbb{N}$, we take a dyadic partition of $C_k$, made of $2^{k\,N}$ cubes with side length $2^{-k}$. We indicate by $C_k(j)$ each of these cubes, with $j=1,\dots,2^{k\,N}$. We also denote by $x_k(j)$ the center of the cube $C_k(j)$ and by
			\[
			S_k:=\Big\{ x_k(i) : 1 \le i \le 2^{k\,N}\Big\},
			\] 
			the collection of all these centers, at a given $k\in\mathbb{N}$.
			Finally, we call \textit{infinite fragile tower} the open set given by 
			\[ 
			\mathcal{T} = \bigcup_{k \in \mathbb{N}} \left( C_k \setminus S_k \right).
			\] 	
			We first show that the condition $d_{\mathcal{T}} \in L^1(\mathcal{T}) \cap L^\infty(\mathcal{T})$ is satisfied. Indeed, we first observe that 
			\[
			r_\mathcal{T}=\frac{5}{12},
			\]
			which implies that $d_\mathcal{T}\in L^\infty(\mathcal{T})$.
			Moreover, we have
			\[
			d_{\mathcal{T}}(x) \le 2^{-k-1}\,\sqrt{N},\qquad \mbox{ for  } x \in C_k(j)\setminus\{x_k(j)\},\ \mbox{ } j=1,\dots,2^{k\,N} \mbox{ and } k\in \mathbb{N},
			\] 
			by construction. Then
			\[
			\begin{split}
				\int_{\mathcal{T}} d_\mathcal{T} \, dx &= \sum_{k \in \mathbb{N}} \sum_{j=1}^{2^{kN}}\int_{C_k(j) \setminus \{x_k(j)\}} d_\mathcal{T} \, dx \\
				&\le \frac{\sqrt{N}}{2} \sum_{k \in \mathbb{N}} \left( \frac{1}{2^k} \right) |C_k \setminus S_k|=\frac{\sqrt{N}}{2} \sum_{k \in \mathbb{N}} \left( \frac{1}{2^k} \right) = \sqrt{N}.   
			\end{split}
			\]
			We now show that for every $1\le q<p\le N$, we have
			\[
			\lambda_{p,q}(\mathcal{T})=0.
			\]
			This would imply that $D^{1,p}_0(\mathcal{T})$ is not continuously embedded in $L^q(\mathcal{T})$.
			At this aim, for every $m\in\mathbb{N}$, we introduce the truncated tower
			\[
			\mathcal{T}_m=\left(\bigcup_{k=0}^m \left( C_k \setminus S_k \right)\right)\setminus (Q\times \{m+1\}).
			\]
			This is a bounded open set contained in $\mathcal{T}$, thus, by monotonicity with respect to set inclusion, we have
			\[
			\lambda_{p,q}(\mathcal{T})\le \lambda_{p,q}(\mathcal{T}_m).
			\]	
			Therefore, in order to get the desired conclusion, it is sufficient to show that 
			\[
			\lim_{m\to\infty} \lambda_{p,q}(\mathcal{T}_m)=0.
			\]
			Since $p\le N$, we know that points have zero $p-$capacity and thus we have (see \cite[Chapter 17]{T})
			\[
			\lambda_{p,q}(\mathcal{T}_m)=\lambda_{p,q}(Q\times (0,m+1)).
			\]
			By appealing to \cite[Main Theorem]{Bra1}, the last quantity can be estimated from above by 
			\[
			\begin{split}
				\lambda_{p,q}(Q\times (0,m+1))&\le \left(\frac{\pi_{p,q}}{2}\right)^p\,\left(\frac{\mathcal{H}^{N-1}(Q\times(0,m+1))}{|Q\times (0,m+1)|^{1-\frac{1}{p}+\frac{1}{q}}}\right)^p\\
				&\le \left(\frac{\pi_{p,q}}{2}\right)^p\,\left(\frac{2\,(N-1)\,(m+1)+2}{(m+1)^{1-\frac{1}{p}+\frac{1}{q}}}\right)^p.
			\end{split}
			\]
			By using that $q<p$, it is easily seen that the last term converges to $0$, as $m$ goes to $\infty$. This gives the desired conclusion.
		\end{enumerate}
		The proof is now over.
	\end{proof}
Before proceeding further, a couple of comments are in order on the geometric estimates obtained in the previous result.	
	\begin{rem}
	\label{rem:vanbasten}
For $p=2$, the lower bound \eqref{eq:pqstima1} has been obtained in \cite[Theorem 3]{vBe}. The proof there is simpler: up to some technical issues, it is simply based on using the trial function $d_\Omega$ in the definition of $\lambda_{2,q}(\Omega)$. However, this produces a poorer estimate: observe that the constant appearing in \cite[equation (9)]{vBe} blows-up as $q\nearrow p=2$.	This is not the case for our estimate \eqref{eq:pqstima1}.
	\end{rem}
	
	\begin{rem}
		Let $N < p <\infty$, $1\le q<p$ and let $\Omega \subseteq \mathbb{R}^N$ be an open set, such that $|\Omega|<\infty$.
		If $d_{\Omega} \in L^{\frac{p\,q}{p-q}}(\Omega)$, then, 
		from the lower bound in \eqref{eq:pqstima2}, we get
		\begin{equation}
		\label{HP2}   
			\lambda_{p,q}(\Omega)\, |\Omega|^{\frac{p-q}{q}} \ge \frac{	\mathfrak{h}_p(\Omega) }{r_{\Omega}^p}.
		\end{equation}
		This is an extension to general open sets of the geometric estimate contained in \cite[Theorem 5.7]{BPZ}. The constant $\mathfrak{h}_p(\Omega)$ is very likely not to be sharp, it would be interesting to determine the sharp constant for \eqref{HP2}.
	\end{rem}
	
	\subsection{The case $p=q$: continuity}
	
	\begin{thm} \label{teo:q=p}  
		Let $1<p<\infty$ and let $\Omega\subsetneq \mathbb{R}^N$ be an open set. The following facts hold:
		\begin{enumerate}
			\item[(i)] we have that  
			\begin{equation}\label{eq:pqstima1pupper}
				\mathcal{D}^{1,p}_0(\Omega)\hookrightarrow L^p(\Omega) \qquad \Longrightarrow \qquad d_{\Omega} \in L^\infty(\Omega),
			\end{equation}
			and the following upper bound holds 
			\begin{equation}\label{eq:pqstima1p}
				\lambda_{p}(\Omega) \,r_\Omega^p \le \lambda_{p}(B_1);
			\end{equation}
			\item[(ii)] moreover, if $N<p<\infty$, then we also have
			\[
			d_\Omega\in L^\infty(\Omega)\qquad \Longrightarrow \qquad\mathcal{D}^{1,p}_0(\Omega)\hookrightarrow L^p(\Omega),
			\]
			and the following lower bound holds
			\begin{equation}\label{HP1} 
				\mathfrak{h}_p(\Omega) \le \lambda_{p}(\Omega) \, r_{\Omega}^{p};
			\end{equation} 
			\item[(iii)] finally, if $p\le N$, then for the open set $\mathcal{P}:=\mathbb{R}^N\setminus \mathbb{Z}^N$ we have
			\[
			d_{\mathcal{P}}\in L^{\infty}(\mathcal{P})\qquad  \mbox{ but }\qquad \mathcal{D}^{1,p}_0(\mathcal{P})\not\hookrightarrow L^p(\mathcal{P}).
			\]
		\end{enumerate}
	\end{thm}
	
	\begin{proof}
		\begin{enumerate}
			\item[(i)] Let $\lambda_p(\Omega)>0$ and let $\{B_{r_n}(x_n)\}_{n \in \mathbb{N}} \subseteq \Omega$ be a sequence of balls such that $r_n$ converges to $ r_\Omega$ as $n$ goes to $\infty$. 
			Thanks to the monotonicity with respect to sets inclusion of $\lambda_{p}$, we get that
			\[ 
			\lambda_{p}(\Omega) \le \lambda_{p}(B_{r_n}(x_n)). 
			\]
			In particular, using the scaling properties of $\lambda_p$, we obtain that 
			\[ r_n^{\,p} \le \frac{\lambda_{p}(B_1)}{ \lambda_{p}(\Omega)},  \]
			and, by sending $n$ to $\infty$, we get $r_{\Omega}<+\infty$ and the upper bound in
			\eqref{eq:pqstima1p};
			\vskip.2cm
			\item[(ii)] let us suppose that $r_{\Omega}<+\infty$ and $N<p<\infty$. By applying the Hardy inequality of Theorem \ref{teo:hardyineq}, we have that 
			\[ 
			\int_{\Omega} |u|^p \, dx \le r^{\,p}_{\Omega} \int_{\Omega} \frac{|u|^p}{d_{\Omega}^p} \, dx \le \frac{1}{\mathfrak{h}_p(\Omega)} \, r_{\Omega}^{\,p} \int_{\Omega} |\nabla u|^p \, dx, \quad \mbox{ for every } u\in C^{\infty}_0(\Omega). 
			\]
			By taking the infimum on $C^{\infty}_0(\Omega)$,
			we get the lower bound in \eqref{HP1}. In particular, if $r_{\Omega}<+\infty$, then $\lambda_p(\Omega)>0$ and thus the continuous embedding $\mathcal{D}^{1,p}_0(\Omega)\hookrightarrow L^p(\Omega)$ holds true;
			\vskip.2cm
			\item[(iii)] it is sufficient to note that
			\[ \lambda_p(\mathcal{P}) \le \lambda_p(B_m \setminus \mathbb{Z}^N), \qquad \mbox{ for every }m\in \mathbb{N}.\]
			Thanks to the assumption $p\le N$, again by \cite[Chapter 17]{T} it holds
			\[
			\lambda_{p}(B_{m} \setminus \mathbb{Z}^N)=\lambda_{p}(B_m).
			\]
			By using the scale property of $\lambda_p$, we get that 
			\[ \lambda_p(\mathcal{P}) \le \lim_{m\to \infty}\lambda_p(B_{m}) = \lim_{m\to \infty} \frac{\lambda_p(B_1)}{m^p}=0 .\] 
			This gives the desired conclusion.
		\end{enumerate}
		\vskip.2cm
		The proof is concluded.
	\end{proof}
\begin{rem}
For $p>N$, the lower bound \eqref{HP1} is an extension to general open sets with finite inradius of the Hersch-Protter-Kajikiya inequality
\[
\lambda_p(\Omega)\ge \left(\frac{\pi_p}{2}\right)^p\,\frac{1}{r_\Omega^p},
\] 
which is valid for every $\Omega\subseteq\mathbb{R}^N$ open {\it convex} set and every $1<p<\infty$ (see \cite{He,Pr} for the case $p=2$ and \cite{Ka1} for the general case). Such an extension can be also found in \cite[Theorem 1.4.1]{Po}, with a different proof and a poorer constant: the result in \cite{Po} is stated for {\it bounded} open sets, however a closer inspection of the proof reveals that it still works for open sets with finite inradius.
\par
Here as well, it would be very interesting to determine the sharp constant $C_{N,p}$ such that for every $\Omega\subseteq\mathbb{R}^N$ open set with finite inradius, we have
\[
\lambda_p(\Omega)\ge \frac{C_{N,p}}{r_\Omega^p},\qquad \mbox{ for every } N<p<\infty.
\] 
We observe that by \eqref{HP1} and \eqref{lowerboundhardyext}, we have 
\[
C_{N,p}\ge \left(\frac{p-N}{p}\right)^p.
\]
\end{rem}

	\subsection{The case $p=q$: compactness}
	
	\begin{thm}\label{teo:compact} Let $1<p<\infty$ and let $\Omega\subsetneq \mathbb{R}^N$ be an open set. The following facts hold: 
		\begin{enumerate}
			\item[(i)] we have that
			\[
			\mathcal D^{1,p}_0(\Omega)\hookrightarrow L^p(\Omega) \mbox{ is compact }\qquad \Longrightarrow \qquad \Omega \mbox{ is }quasibounded; 
			\] 
			\item[(ii)] moreover, if $N<p<\infty$, then we also have
			\[ 
			\Omega \mbox{ is quasibounded }\qquad \Longrightarrow\qquad	\mathcal D^{1,p}_0(\Omega)\hookrightarrow L^p(\Omega) \mbox{ is compact};
			\]
			\item[(iii)] finally, if $p\le N$ and $\mathcal{T}\subsetneq \mathbb{R}^N$ is the same open set of Theorem \ref{teo:q<p}, then $\mathcal{T}$ is quasibounded and the embedding $\mathcal{D}^{1,p}_0(\mathcal{T})\hookrightarrow L^p(\mathcal{T})$ is continuous, but not compact.
		\end{enumerate}
	\end{thm}
	
	\begin{proof}
		\begin{enumerate}
			\item[(i)] This follows from \cite[Example 6.11]{AF}. For completeness, we sketch the idea of the proof: let us suppose that $\Omega$ is not quasibounded. Then there exists a sequence of balls $\{B_r(x_n)\}_{n \in \mathbb{N}} \subseteq\Omega$, with $r>0$ fixed and 
			\[
			\lim_{n\to\infty} |x_n|=+\infty.
			\] 
We consider $\psi \in C^\infty_0(B_1)\setminus\{0\}$ and then we simply set 
\[
\psi_n(x)=\psi\left(\frac{x-x_n}{r}\right),\qquad \mbox{ for } x\in B_r(x_n),\ n\in\mathbb{N}.
\]
It is easily seen that $\{\psi_n\}_{n\in\mathbb{N}}$ is bounded in $\mathcal{D}^{1,p}_0(\Omega)$, but it can not converge in $L^p(\Omega)$;
			\vskip.2cm
			\item[(ii)] this result can be found in \cite[Theorem 2]{A68}, but here we give an alternative proof, which relies on the Hardy inequality of Theorem \ref{teo:hardyineq}. Let $p>N$ and assume that $\Omega$ is quasibounded. By Theorem \ref{teo:q=p}, we already know that $\mathcal{D}^{1,p}_0(\Omega)$ is a functional space, continuously embedded in $L^p(\Omega)$. Let $\{ u_n\}_{n\in\mathbb{N}}\subseteq\mathcal{D}^{1,p}_0(\Omega)$ be a bounded sequence.  We can extend these functions by $0$ outside $\Omega$ and consider them as elements of $W^{1,p}(\mathbb{R}^N)$. In order to apply the classical Riesz–Fréchet–Kolmogorov Theorem, we first observe that by Theorem \ref{teo:q=p} we have that $\{u_n\}_{n\in\mathbb{N}}$ is bounded in $L^p(\mathbb{R}^N)$, as well.
			\par
			Moreover, the bound on the $L^p$ norm of the gradients guarantees that translations converge to $0$ in $L^p(\Omega)$ uniformly in $n$, i.\,e.
			\[
			\lim_{|h|\to 0}\sup_{n\in \mathbb{N}} \int_{\mathbb{R}^N} |u_n(x+h)-u_n(x)|^p \, dx=0 .
			\]
			The crucial point is to exclude the ``loss of mass at infinity''. For this, we exploit the assumption that $\Omega$ is quasibounded.
			The latter entails that for every $\varepsilon>0$, there exists $R>0$ such that 
			\[
			\|d_{\Omega}\|_{L^{\infty}(\Omega\setminus B_R )} <\varepsilon.
			\]
			Let $\eta_R\in C^{\infty}(\mathbb{R}^N)$ be such that
			\[
			0 \le \eta_R \le 1, \quad \eta_R= 1 \mbox{  in }\mathbb{R}^N\setminus B_{R+1},\quad  \eta_R \equiv 0 \mbox{  in } B_R,\quad |\nabla \eta_R | \le C,
			\]
			for some universal constant $C>0$.
			Then 
			\[\sup_{n \in \mathbb{N}} \|\nabla (u_{ n} \eta_R)\|_{L^{p}(\Omega)}\le \sup_{n \in \mathbb{N}} \|\nabla u_n\|_{L^{p}(\Omega)} + C \sup_{n \in \mathbb{N}} \|u_n\|_{L^p(\Omega)}=:M<+\infty.
			\]
			Since the functions $u_n\, \eta_R$ belong to $\mathcal D^{1,p}_0(\Omega)$, by applying H\"older's and Hardy's inequalities (see Theorem \ref{teo:hardyineq} and Remark \ref{rem:estensione}), for every $n\in \mathbb{N}$ we have that
			\[
			\begin{split}
				\left( \int_{\Omega\setminus B_{R+1}} |u_n|^{p} \, dx \right)^\frac{1}{p}
				&\le  \|d_{\Omega}\|_{L^{\infty}(\Omega\setminus B_R)}\,\left( \int_{\Omega} \frac{ |u_n\, \eta_R|^p}{d_{\Omega}^p} \, dx \right)^{\frac{1}{p}} \\   
				&\le \varepsilon \, \mathfrak{h}_p(\Omega)^{-\frac{1}{p}} \left( \int_{\Omega} |\nabla (u_n\, \eta_R) |^{p}\, dx \right)^{\frac{1}{p}} \le \varepsilon \,\mathfrak{h}_p(\Omega)^{-\frac{1}{p}}\, M.
			\end{split}
			\]
			We can thus  appeal to the Riesz–Fréchet–Kolmogorov Theorem and get that, up to a subsequence, $\{ u_n\}_{n\in\mathbb{N}}$ strongly  converges in $L^p(\Omega)$;
			\vskip.2cm	
			\item[(iii)] we consider the set $\mathcal{T}$ defined as in the proof of Theorem \ref{teo:q<p} part (iii). Since $d_{\mathcal{T}} \in L^1(\mathcal{T})$, by applying Lemma \ref{lm:controllo}, we have that  $\mathcal{T}$ is quasibounded. Moreover, the embedding $\mathcal{D}^{1,p}_0(\mathcal{T}) \hookrightarrow L^p(\mathcal{T})$ holds. Indeed, it is sufficient to notice that $\mathcal{T}$ is bounded in one direction. Thus, by Remark \ref{oss:ugualip}, we can infer
			\[
			W^{1,p}_0(\mathcal{T})=\mathcal{D}^{1,p}_0(\mathcal{T}).
			\]
		 However, the embedding $\mathcal{D}^{1,p}_0(\mathcal{T}) \hookrightarrow L^p(\mathcal{T})$ can not be compact.
			Indeed, we take $v\in C^\infty_0(Q \times (0,1))$ not identically zero and we build a bounded sequence $\{v_k\}_{k\in\mathbb{N}}$ by simply translating $v$ in the vertical direction, i.e. for every $k\in\mathbb{N}$ we set
\[
v_k(x',x_N)=v(x',x_N-k),\qquad \mbox{ for every } (x',x_N)\in Q\times (k,k+1).
\]
By appealing again to \cite[Chapter 17]{T}, we have that  
			\[
			\begin{split}
			v_k\in C^\infty_0(Q\times (k,k+1))&\subseteq W^{1,p}_0(Q\times (k,k+1))\\
			&=W^{1,p}_0((Q\times (k,k+1))\setminus S_k) \subseteq W^{1,p}_0(\mathcal{T}) =\mathcal{D}^{1,p}_0(\mathcal{T}),
			\end{split}
			\] 
			for every $k\in\mathbb{N}$. Hence, the sequence $\{v_k\}_{k \in \mathbb{N}}$ is bounded in $\mathcal{D}^{1,p}_0(\mathcal{T})$ and $\|v_k\|_{L^p(\mathcal{T})}>0$ is constant. This shows that $\{v_k\}_{k\in\mathbb{N}}$ can not admit a converging subsequence in $L^p(\mathcal{T})$.
		\end{enumerate}
		\vskip.2cm
		This concludes the proof.
	\end{proof}

	\subsection{The super-homogeneous case $q>p$ and beyond}
	In what follows, for an open set $\Omega\subseteq\mathbb{R}^N$ and for $0<\beta\le 1$, we consider the space 
	\[
	C^{0,\beta}(\overline\Omega)=\Big\{u\in C_{\rm bound}(\overline\Omega)\, :\,[ u ]_{C^{0,\beta}(\overline\Omega)}<+\infty \Big\},
	\] 
	endowed with the standard norm
\[ 
\|u\|_{C^{0,\beta}(\overline\Omega)} = \|u\|_{L^{\infty}(\Omega)} + [ u ]_{C^{0,\beta}(\overline\Omega)},\qquad \mbox{ for every }u\in C^{0,\beta}(\overline\Omega). 
\]
As a consequence of the previous embedding results, we can draw the following picture, for the case $N<p<q$. The proof is essentially an exercise. 
	\begin{cor}
	\label{cor:embeddings}
	Let $p>N$ and let $\Omega\subsetneq\mathbb{R}^N$ be an open set. The following facts hold:
	\vskip.2cm
	\begin{enumerate}
		\item[(i)] if $d_{\Omega} \in L^{\infty}(\Omega)$, then we have
		\[
		\mathcal{D}^{1,p}_0(\Omega)\hookrightarrow L^q(\Omega),\qquad \mbox{ for every } p\le q\le \infty,
		\]
		and
		\[
		\mathcal{D}^{1,p}_0(\Omega)\hookrightarrow C_0(\Omega)\cap C^{0,\beta}(\overline\Omega),\qquad \mbox{ for every }0<\beta\le \alpha_p;\]

		\item[(ii)] if $\Omega$ is quasibounded, then the above embeddings are compact, for 
		\[
		p\le q\le \infty\qquad \mbox{ and }\qquad 0<\beta < \alpha_p;
		\]
		\item[(iii)] if $d_\Omega \in L^\gamma(\Omega)$, for some $1 \le \gamma < \infty$, then we have
		\[
		\mathcal{D}^{1,p}_0(\Omega)\hookrightarrow L^q(\Omega),\qquad \mbox{ for every } \frac{p\,\gamma}{p+\gamma}\le q\le \infty,
		\]
		and such an embedding is compact.
	\end{enumerate}
\end{cor}

\begin{proof}
	\begin{enumerate}
		\item[(i)] Let $d_{\Omega} \in L^{\infty}(\Omega)$. The existence of the embedding $\mathcal{D}^{1,p}_0(\Omega)\hookrightarrow L^p(\Omega)$ is a consequence of Theorem \ref{teo:q=p} part (ii). By using the Gagliardo-Nirenberg interpolation inequality \eqref{GNS} with $q=p$, it follows that $\mathcal{D}^{1,p}_0(\Omega)$ is continuously embedded in every $L^q(\Omega)$ with $p \le q \le \infty$.
\par
As for the embedding in H\"older spaces: we observe at first that from the embedding $\mathcal{D}^{1,p}_0(\Omega)\hookrightarrow L^\infty(\Omega)$, we obtain that each $\{u_n\}_{n\in\mathbb{N}}\subset C^\infty_0(\Omega)$ which is a Cauchy sequence in the norm of $\mathcal{D}^{1,p}_0(\Omega)$, is a Cauchy sequence in the sup norm, as well. Thus, by recalling the definition of the completion space $C_0(\Omega)$, we get that $\mathcal{D}^{1,p}_0(\Omega)\hookrightarrow C_0(\Omega)$. By using this fact
and Corollary \ref{cor:morrey}, we thus get that $\mathcal{D}^{1,p}_0(\Omega)$ is continuously embedded in $C_0(\Omega)\cap C^{0,\alpha_p}(\overline\Omega)$. Then Lemma \ref{lm:interpol_holder} gives the desired conclusion;
		\vskip.2cm
		\item[(ii)] we now suppose that $\Omega$ is quasibounded. In order to prove the first statement, it is sufficient to observe that the embedding $\mathcal{D}^{1,p}_0(\Omega)\hookrightarrow L^q(\Omega)$ is compact for $q=p$ thanks to Theorem \ref{teo:compact} part (ii). By applying again the Gagliardo-Nirenberg inequality \eqref{GNS} with $q=p$, we conclude. 
\par
The case of $C_0(\Omega)\cap C^{0,\beta}(\overline\Omega)$ follows as above, by combining Morrey's inequality and Lemma \ref{lm:interpol_holder};		
			\vskip.2cm
		\item[(iii)] we first recall that the assumption $d_\Omega \in L^\gamma(\Omega)$, for some $1 \le \gamma < \infty$, implies that $\Omega$ is a quasibounded set (see Lemma \ref{lm:controllo}). The compact embedding $\mathcal{D}^{1,p}_0(\Omega)\hookrightarrow L^q(\Omega)$ easily follows by Theorem \ref{teo:q<p} part (ii), when $q=p\,\gamma/(p+\gamma)$, while the case $q=\infty$ was just proved in the part (ii) above. We conclude, by interpolation, that the embedding is compact for every $p\,\gamma/(p+\gamma)\le q\le \infty$.
	\end{enumerate}
	\vskip.2cm
The proof is now complete
\end{proof}
\begin{rem}
It is not difficult to see that the compact embedding of Corollary \ref{cor:embeddings} part (ii) {\it does not} extend up to the borderline case $\beta=\alpha_p=1-N/p$. This can be seen by means of a standard scaling argument: take $\Omega=B_1$ and $\psi\in C^\infty_0(B_1)\setminus\{0\}$. We define the sequence 
\[
\psi_n(x)=n^\frac{N-p}{p}\,\psi(n\,x),\qquad \mbox{ for } n\in\mathbb{N}.
\]
It is easily seen that 
\[
\|\nabla \psi_n\|_{L^p(B_1)}=\|\nabla \psi\|_{L^p(B_1)}\qquad \mbox{ and }\qquad [\psi_n]_{C^{0,\alpha_p}(B_1)}=[\psi]_{C^{0,\alpha_p}(B_1)}.
\]
On the other hand, by construction,  we have that $\psi_n$ converges uniformly to $0$ as $n$ goes to $\infty$, since $N-p<0$. Thus, for this sequence we can not have convergence in the norm of $C^{0,\alpha_p}(\overline{B_1})$.
\end{rem}	
We complete the previous result by giving some geometric estimates for the generalized principal frequencies $\lambda_{p,q}$ in the case $N<p<q$, as well. 
	\begin{cor}
		\label{coro:endpoint}
		Let $N<p< \infty$, $p\le q\le \infty$ and let $\Omega\subsetneq\mathbb{R}^N$ be an open set. We have that 
		\[ 
		\lambda_{p,q}(\Omega)>0 \qquad \Longleftrightarrow \qquad r_{\Omega}<+\infty,
		 \]
and 	
		\begin{equation}
			\label{eq:boundslambdapinfty}
			\frac{\mathfrak{h}_{p,q}(\Omega)}{r_\Omega^{p-N+N\,\frac{p}{q}}}\le \lambda_{p,q}(\Omega)\le \frac{\lambda_{p,q}(B_1)}{r_\Omega^{p-N+N\,\frac{p}{q}}},
		\end{equation}
	with $\mathfrak{h}_{p,q}(\Omega)$ defined in Theorem \ref{teo:hardyineq}.
	Moreover, if $\Omega$ is quasibounded, then there exists $u_{p,q} \in W^{1,p}_0(\Omega)$ which solves
		\begin{equation}
		\label{lambdapinfty}	
		\lambda_{p,q}(\Omega)=\inf_{u\in W^{1,p}_0(\Omega)} \left\{\int_\Omega |\nabla u|^p\,dx\, :\, \|u\|_{L^q(\Omega)}=1\right\}.
		\end{equation}
	 	\end{cor}
	\begin{proof}
	 Let us assume $\lambda_{p,q}(\Omega)>0$ and let $\{B_{r_n}(x_n)\}_{n \in \mathbb{N}} \subseteq \Omega$ be a sequence of balls such that $r_n$ goes to $r_\Omega$, as $n$ goes to $\infty$. As in the proof of Theorem \ref{teo:q=p} part (i), it follows that 
		\[ 
		r_n^{\,p-N+N\frac{p}{q}} \le \frac{\lambda_{p,q}(B_1)}{ \lambda_{p,q}(\Omega)}, 
		 \]
		and, by sending $n$ to $\infty$, we get $r_{\Omega}<+\infty$ and the upper bound in \eqref{eq:boundslambdapinfty}. 
		\par
		In order to prove the reverse implication, we first observe that this has already been proved in Theorem \ref{teo:q=p} for the case $q=p$ part (ii). For the case $p<q\le \infty$, it is sufficient to use the same argument, in conjuction with the general Hardy inequality of Theorem \ref{teo:hardyineq}. This comes with the lower bound in \eqref{eq:boundslambdapinfty}. We leave the details to the reader. 
\par
We now come to the existence part, under the stronger assumption that $\Omega$ is quasibounded. We first observe that the identity \eqref{lambdapinfty} follows from Lemma \ref{lm:inf}. Moreover, the assumption on $\Omega$, Theorem \ref{teo:q=p} and Remark \ref{oss:ugualip} guarantee that we have 
\[
\mathcal{D}^{1,p}_0(\Omega)=W^{1,p}_0(\Omega),
\]
thanks to Proposition \ref{prop:spazi0}. The existence of a minimizer is now an easy consequence of the Direct Method in the Calculus of Variations, once observed that $W^{1,p}_0(\Omega)$ is weakly closed and that we have the compact embeddings of Corollary \ref{cor:embeddings} at our disposal.
	\end{proof}
	\begin{rem}\label{rem:hyndlind}
		We notice that the value of $\lambda_{p,\infty}(B_1)$ can be made explicit: according to \cite[Theorem 2E]{Tal} we have 
		\[
		\lambda_{p,\infty}(B_1)=\left(\frac{p-N}{p-1}\right)^{p-1}\,N\,\omega_N,\qquad \mbox{ for } p>N.
		\]
		This implies that the upper bound for the sharp Morrey constant in \eqref{eq:mp} can be rewritten as
		\[ \mathfrak{m}_{p}(\mathbb{R}^N)\le \lambda_{p, \infty}(B_1). \] 
		Moreover, such a value is uniquely attained by the functions
		\[
			u(x)=\pm \left(1-|x|^\frac{p-N}{p-1}\right)_+.
		\]
		We refer to \cite{EP, HL} for a thorough study of the variational problem associated to $\lambda_{p,\infty}$, in the case of bounded sets.
	\end{rem}

	\section{Asymptotics}\label{Sec:asymp}
	
	\subsection{Asymptotics for $\lambda_{p,q}(\Omega)$}
	
	\begin{cor}\label{teo:limite}  
		Let $1\le q< \infty$ and let $\Omega\subsetneq \mathbb{R}^N$ be an open set. Then 
		\[
		\begin{split} 
		\lim_{p \to \infty} \Big( \lambda_{p,q}(\Omega) \Big)^\frac{1}{p} = \frac{1}{\|d_\Omega\|_{L^q(\Omega)}},
		\end{split}
		\]
		and 
		\[
		\lim_{p\to\infty} \Big(\lambda_{p,\infty}(\Omega)\Big)^\frac{1}{p}=\frac{1}{r_{\Omega}}.
		\]
		In the previous equations, the right-hand sides have to be considered $0$, if $d_\Omega\not\in L^q(\Omega)$ or $r_\Omega=+\infty$, respectively.
	\end{cor}
	
	\begin{proof}
		We start with the case $q=\infty$. If $r_{\Omega}=+\infty$, thanks to Corollary \ref{coro:endpoint}, there is nothing to prove. Let us assume $r_{\Omega}<+\infty$, it is sufficient to take the $p-$rooth in \eqref{eq:boundslambdapinfty} and use that
		\[ 
		\lim_{p \to \infty} \Big(\lambda_{p,\infty}(B_1)\Big)^\frac{1}{p} = 1, 
		\]
		(see Remark \ref{rem:hyndlind}) and \eqref{hardyinfty}. This gives the desired conclusion as $p$ goes to $\infty$.
		\vskip.2cm\noindent
		We now consider the case $q<\infty$. We first suppose that $d_\Omega\in L^q(\Omega)$. Observe that for every $p>2\,q$, we have 
		\[
		d_{\Omega}(x)^{\frac{p\,q}{p-q}} \le r_{\Omega}^{\frac{q^2}{p-q}}\, d_{\Omega}(x)^q\le \Big(\max\{1,r_\Omega\}\Big)^q\,d_\Omega(x)^q,\qquad \mbox{ for every } x\in\Omega,
		\] 
		thus we can apply the Dominated Convergence Theorem to get that 
		\begin{equation}
			\label{convpdist} 
			\lim_{p\to \infty} \left(\int_{\Omega} d_{\Omega}^{\frac{p\,q}{p-q}} \, dx\right)^\frac{p-q}{q\,p}=\left(\int_{\Omega} d^{\,q}_{\Omega}\, dx\right)^\frac{1}{q}.
		\end{equation}
		Moreover, by Theorem \ref{teo:q<p}, for every $p>q$ and $p>N$, we have the two-sided estimate 
		\[
		\mathfrak{h}_p(\Omega)\le \lambda_{p,q}(\Omega) \,\left(\displaystyle\int_{\Omega} d_{\Omega}^{\frac{p\,q}{p-q}} \, dx \right)^{\frac{p-q}{q}} \le \lambda_p(B_1).
		\]
		By raising this estimate to the power $1/p$, using \eqref{hardyinfty}, \eqref{convpdist} and the following fact
		\[
		\lim_{p\to\infty} \Big(\lambda_p(B_1)\Big)^\frac{1}{p}=1,
		\] 
		(see \cite[Lemma 1.5]{JLM}), we get the desired conclusion.
		\par
		We now suppose that $d_\Omega \notin L^q(\Omega)$. Let $n_0 \in \mathbb{N}$ such that  $\Omega_n:=\Omega\cap B_n \ne \emptyset$ for every $n \ge n_0$. By applying the first part of this proof to the set $\Omega_n$ with $n\ge n_0$, we have that 
		\[
		\lim_{p \to \infty} \Big( \lambda_{p,q}(\Omega_n) \Big)^\frac{1}{p} =\frac{1}{\| d_{\Omega_n} \|_{L^q(\Omega_n)}}.
		\]
		Hence,  by using the  monotonicity of $\lambda_{p,q}$ with respect to the set inclusion,  we get that 	
		\begin{equation}
			\label{eq:lapR}
			\limsup_{p \to \infty} \Big( \lambda_{p,q}(\Omega) \Big)^\frac{1}{p} \le  \frac{1}{\| d_{\Omega_n} \|_{L^q(\Omega_n)}},
		\end{equation}
		for every $n\ge n_0$. 
		We extend each distance function $d_{\Omega_n}$ equal to $0$ in $\mathbb{R}^N\setminus \Omega_n$.
		We note that the family $\{d_{\Omega_n} \}_{n \ge n_0}$ is not decreasing with respect to $n$. Thus, in order to conclude, it is sufficient to prove  that
		\begin{equation}\label{supdR}
			\lim_{n \to \infty} d_{\Omega_n}(x) = d_\Omega(x),\qquad \mbox{ for every } x\in\Omega.
		\end{equation}
		Indeed, by passing to the limit in \eqref{eq:lapR} as $n$ goes to $\infty$ and by using Monotone Convergence Theorem, we get that  
		\[
		\limsup_{p \to \infty} \Big( \lambda_{p,q}(\Omega) \Big)^\frac{1}{p} =0.\] 
		In order to show \eqref{supdR}, 
		we note that, for every $x\in \Omega$, there exists $n_x\ge n_0$ such that $B_{d_{\Omega}(x)}(x) \subseteq\Omega_n$, for every $n\ge n_x$. This implies that  
		\[ 
		d_{\Omega}(x) = d_{\Omega_n}(x), \quad \mbox{ for every } n \ge n_x.
		\]
		This concludes the proof.	
	\end{proof} 

	\begin{cor}\label{teo:limiteq}  
		Let $N< p< \infty$ and let $\Omega\subsetneq\mathbb{R}^N$  be an open set. Then 
		\[ \lim_{q \to \infty}  \lambda_{p,q}(\Omega) = \lambda_{p,\infty}(\Omega).\]		
	\end{cor}
	
	\begin{proof}
	Let $\psi\in C^\infty_0(\Omega)\setminus\{0\}$.  By definition of $\lambda_{p,q}(\Omega)$ we have that, for every $q\geq p$, it holds
	\[
\lambda_{p,q}(\Omega)\le \frac{\displaystyle\int_\Omega |\nabla \psi|^p\,dx}{\left(\displaystyle\int_\Omega |\psi|^q\,dx\right)^\frac{p}{q}}.
\]
If we now take the limit as $q$ goes to $\infty$, we get
\[
\limsup_{q\to \infty} \lambda_{p,q}(\Omega)\le \lim_{q\to \infty }\frac{\displaystyle\int_\Omega |\nabla \psi|^p\,dx}{\left(\displaystyle\int_\Omega |\psi|^q\,dx\right)^\frac{p}{q}}=\frac{\displaystyle\int_\Omega |\nabla \psi|^p\,dx}{\|\psi\|^{\,p}_{L^\infty(\Omega)}}.
\]
By arbitrariness of $\psi$ and recalling the definition of $\lambda_{p,\infty}(\Omega)$, we obtain
\[
\limsup_{q\to \infty} \lambda_{p,q}(\Omega)\le \lambda_{p,\infty}(\Omega).
\]
In order to show the converse inequality, we can assume that  $\lambda_{p,\infty}(\Omega)>0$, otherwise from the previous inequality we already get the desired result. Since $p>N$, by Corollary \ref{coro:endpoint} we have that $r_{\Omega}<+\infty$ and thus $\lambda_p(\Omega)>0$, as well.
Then, for every $u \in C^\infty_0(\Omega)\setminus\{0\}$ and for every $p< q<\infty$ it holds
		\[
		\|u\|_{L^q(\Omega)}\le \|u\|_{L^{p}(\Omega)}^\frac{p}{q}\,\|u\|_{L^\infty(\Omega)}^{1-\frac{p}{q}}\le
		\Big(\lambda_{p}(\Omega)\Big)^{-\frac{1}{q}}\,\|\nabla u\|_{L^p(\Omega)}^\frac{p}{q}\,\|u\|_{L^\infty(\Omega)}^{1-\frac{p}{q}},
		\]
		by interpolation in Lebesgue spaces. Hence, we have the following lower bound
		\[
		\begin{split}
			\frac{\|\nabla u\|_{L^p(\Omega)}}{ \|u\|_{L^q(\Omega)} }\ge \Big( \lambda_{p}(\Omega)\Big)^{\frac{1}{q}}\,\frac{\|\nabla u\|^{1-\frac{p}{q}}_{L^p(\Omega)} }{\|u\|_{L^\infty(\Omega)}^{1-\frac{p}{q}}} \ge \Big( \lambda_{p}(\Omega)\Big)^{\frac{1}{q}}\,\Big(\lambda_{p,\infty}(\Omega)\Big)^{\frac{q-p}{p\,q}}.
		\end{split}
		\]
		By raising to the power $p$ on both sides and taking the infimum on $C^\infty_0(\Omega)\setminus\{0\}$ on the left-hand side, this yields
		\[
			\lambda_{p,q}(\Omega)\ge \Big(\lambda_{p}(\Omega)\Big)^{\frac{p}{q}}\,\Big(\lambda_{p,\infty}(\Omega)\Big)^{\frac{q-p}{q}}. 
		\]
By sending $q$ to $\infty$ in this inequality, we get
\[
\liminf_{q\to \infty} \lambda_{p,q}(\Omega)\ge \lambda_{p,\infty}(\Omega),
\]
as desired.
\end{proof}

	\subsection{Asymptotics for the solution of the Lane-Emden equation}
	
Let $\Omega \subsetneq \mathbb{R}^N$ be an open connected set. In this subsection we will assume that $1\le q<\infty$ and $d_{\Omega} \in L^q(\Omega)$. By Lemma \ref{lemma:r<infty}, the last assumption entails that $d_{\Omega} \in L^{\infty}(\Omega)$, as well. Thus, by interpolation, we have that $d_{\Omega}\in L^{(p\,q)/(p-q)}(\Omega)$ for every $q<p<\infty$. Hence, by Theorem \ref{teo:q<p}, $\Omega$ is $(p,q)-$admissible for every $p>\max\{N,q\}$ and, by \cite[Corollary 4.4]{BPZ}, there exists a unique positive solution $w_{p,q}^\Omega $ to \eqref{eq:cauchy} for every $q<p$. In the following theorem, we will study the asymptotic behavior of $w_{p,q}^{\Omega}$, as $p$ goes to $\infty$. 
	
	\begin{thm}\label{teo:asymp}
		Let $1\le q < \infty$ and let $\Omega \subsetneq \mathbb{R}^N$ be an open connected set such that $d_{\Omega}\in L^q(\Omega)$. Then 
		\begin{equation}
			\label{convw}
			\lim_{p\to\infty} \|w_{p,q}^{\Omega}-d_\Omega\|_{L^r(\Omega)}=0\qquad \mbox{ and }\qquad \lim_{p\to\infty} \|w_{p,q}^{\Omega}-d_\Omega\|_{C^{0,\beta}(\overline\Omega)}=0
		\end{equation}
		for every $q\le r\le \infty$ and every $0<\beta<1$.
	\end{thm}
	
	\begin{proof}
		We will first show that \eqref{convw} holds to $r=q$. Then, by interpolation, we will obtain all the other claimed convergences.	
		\vskip.2cm\noindent   
		{\it Part 1: convergence in $L^{q}(\Omega)$.}
		We extend each function $w_{p,q}^\Omega $  to $\mathbb{R}^N$ by setting it to be zero in $\mathbb{R}^N\setminus \Omega$.
		First of all, we note that, by using \eqref{pqnorm} and Corollary \ref{teo:limite}, we have 
				\begin{equation}
			\label{pgradnorm}  
			\lim_{p\to \infty}\int_\Omega |\nabla w_{p,q}^{\Omega}|^{p} \, dx =\lim_{p\to\infty} \int_\Omega |w_{p,q}^{\Omega}|^{q} \, dx = \lim_{p\to\infty} \left (\frac{1}{\Big(\lambda_{p,q}(\Omega)\Big)^\frac{1}{p}}\right)^\frac{p\,q}{p-q}=\int_\Omega d_\Omega^{\,q}\,dx,
		\end{equation}
		which implies 
		\begin{equation}\label{eq:bound1}
			\lim_{p\to \infty} \|\nabla w_{p,q}^\Omega | |_{L^p(\Omega)} = 1.
		\end{equation}
				Moreover, by applying \eqref{eq:morreylemma4.1}, we find the upper bound   
		\begin{equation}\label{eq:up}
			0< w_{p,q}^\Omega (x) \le \frac{d_{\Omega}(x)^{\alpha_ p}}{\Big(\mu_p(B_1)\Big)^{\frac{1}{p}}}\, \|\nabla w_{p,q}^\Omega \|_{L^p(\Omega)}, \qquad \mbox{ for every } x\in\Omega,
		\end{equation}
		where $\alpha_p=1-N/p$. On the other hand, thanks to \eqref{eq:lowest}, we obtain the lower bound
		\begin{equation}\label{eq:low}
			(d_{\Omega}(x))^{\frac{p}{p-q}} \, w_{p,q}^{B_1}(0) \le w_{p,q}^{\Omega}(x), \quad \mbox{ for every } x \in \Omega. 
		\end{equation}
		By sending $p$ to  $\infty$ in \eqref{eq:up} and \eqref{eq:low} and taking into account \eqref{eq:limitemup}, \eqref{eq:bound1} and \eqref{eq:behaviorpalla}, we get that 
		\[ 
		\lim_{p\to \infty}w_{p,q}^\Omega (x) = d_{\Omega}(x), \quad \mbox{ for every } x\in\Omega.
		\]
		The  pointwise convergence, combined with the convergence of the $L^q$ norm given by  \eqref{pgradnorm},  implies that
		\[
		\lim_{p\to\infty} \|w_{p,q}^\Omega -d_\Omega\|_{L^q(\Omega)}=0.
		\]
		{\it Part 2: convergence in $L^{\infty}(\Omega)$.}
By Corollary \ref{cor:embeddings}, we have that $w^\Omega_{p,q}\in C_0(\Omega)\cap C^{0,\alpha_p}(\overline{\Omega})$. 
	Moreover, by applying the estimate on the sharp Morrey constant of Corollary \ref{cor:morrey}, we have  that $w_{p,q}^\Omega$ satisfies 
	\[
	[w^\Omega_{p,q}]_{C^{0,\alpha_p}(\overline\Omega)}\le \left(\frac{1}{\mathfrak{m}_p(\Omega)}\right)^\frac{1}{p}\,\|\nabla w^\Omega_{p,q}\|_{L^p(\Omega)}.
	\]
	By using \eqref{eq:bound1} and Corollary \ref{cor:morrey}, we have 
	\[
	\limsup_{p\to\infty}\,[w^\Omega_{p,q}]_{C^{0,\alpha_p}(\overline\Omega)}\le 1,
	\]
	and thus in particular the seminorms $[w^\Omega_{p,q}]_{C^{0,\alpha_p}(\overline\Omega)}$ are uniformly bounded, for $p$ large enough. We also observe that by Lemma \ref{lm:distholder}, we have
	\[
	[d_\Omega]_{C^{0,\alpha_p}(\overline\Omega)}\le (2\,r_\Omega)^{1-\alpha_p}.
	\]
We now apply Lemma \ref{lm:interpol_holder} to $w^\Omega_{p,q}-d_\Omega$, with $\alpha=\alpha_p$ and $\gamma=q$. Thus, for every $0<\beta<1$ and every $p$ such that $\alpha_p>\beta$, we have 
\[
[w^\Omega_{p,q}-d_\Omega]_{C^{0,\beta}(\overline\Omega)}\le C_1\,\|w^\Omega_{p,q}-d_\Omega\|_{L^q(\Omega)}^{\theta_p}\,[w^\Omega_{p,q}-d_\Omega]^{1-\theta_p}_{C^{0,\alpha_p}(\overline\Omega)},\qquad \mbox{ with } \theta_p=\frac{\alpha_p-\beta}{\alpha_p+\dfrac{N}{q}},
\]
and
\[
\|w^\Omega_{p,q}-d_\Omega\|_{L^\infty(\Omega)}\le C_2\,\|w^\Omega_{p,q}-d_\Omega\|_{L^q(\Omega)}^{\chi_p}\,[w^\Omega_{p,q}-d_\Omega]^{1-\chi_p}_{C^{0,\alpha_p}(\overline\Omega)},\qquad \mbox{ with } \chi_p=\frac{\alpha_p}{\alpha_p+\dfrac{N}{q}}.
\]
We observe that, for $p$ diverging to $\infty$, the exponent $\alpha_p$ goes to $1$. Thus,  the constants $C_1$ and $C_2$, which depend on $p$ through $\alpha_p$, stay uniformly bounded as $p$ goes to $\infty$ (see Lemma \ref{lm:interpol_holder} and Remark \ref{rem:constantvaries}). By using this fact, the bound on the $C^{0,\alpha_p}$ seminorms inferred above and the convergence in $L^q$ proved in {\it Part 1}, the previous interpolation estimates give
\[
\lim_{p\to\infty} \left(\|w^\Omega_{p,q}-d_\Omega\|_{L^\infty(\overline\Omega)}+[w^\Omega_{p,q}-d_\Omega]_{C^{0,\beta}(\overline\Omega)}\right)=0.
\]
Finally, the convergence in $L^r(\Omega)$ for $q<r<\infty$ can be obtained by interpolation in Lebesgue spaces.
	\end{proof}

\subsection{Asymptotics for $\lambda_{p}(\Omega)$}
The following corollary generalizes the result shown, independently, in \cite[Theorem 3.1]{FIN} and \cite[Lemma 1.2]{JLM}. While these treat the case of {\it bounded} open sets, we enlarge the result to cover {\it every} open set, without further restrictions.
\begin{cor}\label{teo:asymp-p}
	Let $\Omega\subsetneq\mathbb{R}^N$ be an open set, then
	\begin{equation}\label{teo:limitepp} 
		\lim_{p \to \infty} \Big( \lambda_{p}(\Omega) \Big)^\frac{1}{p} =  \frac{1}{r_\Omega},
	\end{equation}	
	where the right-hand side has to be considered $0$, if $r_\Omega=+\infty$.
\end{cor}
\begin{proof} 
	First of all, we note that for every $0<r<r_{\Omega}$  there exists a ball $B_r(x_r)\subseteq\Omega$. Hence, by applying \cite[Lemma 1.5]{JLM}, it holds 
	\begin{equation}\label{liminfpp} 
		\limsup_{p \to \infty} \Big( \lambda_{p}(\Omega) \Big)^\frac{1}{p} \le  \limsup_{p \to \infty} \Big( \lambda_{p}(B_r(x_r)) \Big)^\frac{1}{p} =\frac{1}{r}. 	
	\end{equation}
	By sending $r \to r_{\Omega}$, we get 
	\[
	\limsup_{p \to \infty} \Big( \lambda_{p}(\Omega) \Big)^\frac{1}{p} \le \frac 1 {r_{\Omega}},
	\]
	where the right-hand side is $0$ when $r_{\Omega}=+\infty$.
	In order to obtain the reverse inequality when $r_{\Omega}<+\infty$, it is sufficient to apply \eqref{HP1} and \eqref{hardyinfty}, to get that
	\[ 
	\liminf_{p \to \infty} \Big( \lambda_{p}(\Omega) \Big)^\frac{1}{p} \ge \frac{1}{r_{\Omega}} \, \lim_{p \to \infty}\Big(\mathfrak{h}_p(\Omega)\Big)^{\frac{1}{p}} = \frac{1}{r_{\Omega}}. 
	\]
	This concludes the proof.
\end{proof}

\subsection{Asymptotics for the first $p-$eigenfunction}
We first recall that for every function $u \in W^{1,\infty}(\Omega)$ vanishing on the boundary $\partial\Omega$, we have 
\[ 
|u(x)| \le d_{\Omega}(x) \, \|\nabla u\|_{L^{\infty}(\Omega)}, \quad \mbox{ for every } x \in \Omega. 
\]
This may be seen as a limit case of Hardy's inequality.
In particular $\pm \,d_{\Omega}/r_{\Omega}$ is a solution of the following minimization problem
\begin{equation}\label{mininfty}
	\min_{u \in W^{1,\infty}(\Omega)} \Big\{\|\nabla u \|_{L^{\infty}(\Omega)} :\, \| u \|_{L^{\infty}(\Omega)}=1, \, u \equiv 0 \mbox{ on } \partial \Omega \Big\}=\frac{1} {r_{\Omega}},
\end{equation}
provided $\Omega$ has finite inradius.
\begin{thm}\label{teo:asymptup}
	Let $\Omega\subsetneq \mathbb{R}^N$ be an open connected quasibounded set. Then, for every $N<p<\infty$, there exists a unique positive solution   $u_p$  of the  problem
	\begin{equation}\label{eq:minprobllambdap}
	\lambda_p(\Omega)=\min_{u\in W^{1,p}_0(\Omega)}\left\{\int_\Omega |\nabla u|^p\,dx\, :\, \int_\Omega |u|^p\,dx=1\right\}.
	\end{equation}
	
	Moreover,  the family $\{ u_p\}_{p>N}$ is precompact in $C^{0,\beta}(\overline\Omega)$ for every $0<\beta<1$ and  every accumulation point $u_{\infty}$ is a solution of \eqref{mininfty}, possibly different from $d_{\Omega}/r_{\Omega}$. 
\end{thm}
\begin{proof}
We first observe that $\lambda_p(\Omega)>0$, thanks to the assumption on $\Omega$. Thus, by Remark \ref{oss:ugualip},  we have
	\[
	W^{1,p}_0(\Omega)=\mathcal{D}^{1,p}_0(\Omega).
	\]
	By Theorem \ref{teo:compact},  for every $p>N$ the embedding $W^{1,p}_0(\Omega)\hookrightarrow L^p(\Omega)$ is compact. Hence, by using also Lemma \ref{lm:inf}, it follows that, for every $p>N$,  there exists a positive solution $u_{p}\in W^{1,p}_0(\Omega)$ of  the minimization problem \eqref{eq:minprobllambdap}. Uniqueness can now be inferred by using the {\it Benguria hidden convexity principle} of \cite{BK, KLP}, for example, as generalized in \cite[Theorem 2.9]{BPZ}. See also \cite{AH} and \cite{Lin} for other proofs of the uniqueness.
	\par
	Without loss of generality, let $\{p_n\}_{n \in \mathbb{N}}$ be an increasing sequence diverging at $\infty$. In particular, there exists $n_0\in\mathbb{N}$ such that $p_n>N$ for every $n\ge n_0$. We denote by $u_{p_n}$ the unique positive solution of the problem defining $\lambda_{p_n}(\Omega)$. By applying \eqref{eq:morreylemma4.1} and \eqref{eq:morrey1*},  we find that
	\begin{equation}\label{eq:upn} 
		|u_{p_n}(x)|\le \frac{1}{\Big(\mu_{p_n}(B_1)\Big)^{\frac{1}{ p_n}}}\Big(\lambda_{p_n}(\Omega)\Big)^{\frac 1{p_n}}   d_{\Omega}(x)^{\alpha_{p_n}},\quad \mbox{ for every } x\in \overline{\Omega},
	\end{equation}
	and
	\begin{equation}\label{eq:limupn} 
		|u_{p_n}(x)-u_{p_n}(y)| \le \frac{1}{\Big(\mu_{p_n}(B_1)\Big)^{\frac 1 {p_n}}}\Big(\lambda_{p_n}(\Omega)\Big)^{\frac 1{p_n}} |x-y|^{\alpha_{p_n}}, \quad \mbox{ for every } x,y\in  \overline{\Omega}.
	\end{equation}
	Since $\Omega$ is quasibounded, the previous estimates assures that we can apply Proposition \ref{prop:AscArz}.
	Thus, we get that $\{u_ {p_n}\}_{n\ge n_0}$ converges uniformly to a function $u_{\infty}\in C_0(\Omega)$, up to a subsequence. By passing to the limit in \eqref{eq:upn} and in \eqref{eq:limupn} as $n$ goes to $\infty$, we obtain that
	\[|u_{\infty}(x)|\le \frac{1}{r_{\Omega}}\, d_{\Omega}(x), \quad \mbox{ for every } x\in \overline \Omega,\]
	and
	\[|u_{\infty}(x)-u_{\infty}(y)| \le \dfrac{1}{r_{\Omega}} \, |x-y|, \quad \mbox{ for every } x,y\in \overline \Omega.\]
	Observe that we also used \eqref{eq:limitemup} and \eqref{teo:limitepp}.
	In particular $u_{\infty} \in W^{1,\infty}(\Omega)$ and it satisfies 
	\[ 
	\|u_\infty\|_{L^\infty(\Omega)}\le 1 \qquad \mbox{ and } \qquad \|\nabla u_\infty\|_{L^\infty(\Omega)}\le  \dfrac{1}{r_{\Omega}}.
	 \]
	If we also prove that $\|u_{\infty}\|_{L^\infty(\Omega)}\ge 1$, we can conclude that $u_{\infty}$ is a minimizer for the problem \eqref{mininfty}.
	\par	In order to show this, for every $R>0$, we take $\eta_R\in C^{\infty}(\mathbb{R}^N)$ such that
	\[0 \le \eta_R \le 1,\quad  \eta_R= 1 \mbox{  in }\mathbb{R}^N\setminus B_{R+1},\quad \eta_R= 0\mbox{  in } B_R,\quad |\nabla \eta_R | \le C, \]
	for some universal constant $C>0$.
		Then 
		\[
		\begin{split}
			\sup_{n\ge n_0} \|\nabla (u_{ p_n} \eta_R)\|_{L^{p_n}(\Omega)} &\le \sup_{n\ge n_0} \|\nabla u_{p_n}\|_{L^{p_n}(\Omega)} + C \,\sup_{n>N} \|u_{p_n}\|_{L^{p_n}(\Omega)} \\
			&=\sup_{n\ge n_0} \Big(\lambda_{p_n}(\Omega)\Big)^{\frac{1}{p_n}} +C\\
			&\le \frac{1}{r_{\Omega}} \, \sup_{n\ge n_0}  \Big(\lambda_{p_n}(B_1)\Big)^{\frac 1 {p_n}} +C=:M <+\infty.
		\end{split}
		\]
		We notice that $M>0$ only depends on $N$ and $r_{\Omega}$. Now, thanks to \eqref{hardyinfty}, there exists $\overline{p}>N$ such that 
	\[ 
	\Big(\mathfrak{h}_{p}(\Omega)\Big)^{\frac{1}{p}}\ge \frac{1}{2},\qquad \mbox{ for every } p\ge \overline{p}.
	\]
	Since $\Omega$ is quasibounded, for every $0<\varepsilon\le 1/2$, there exists $R_{\varepsilon}>0$ such that 
	\[
	\|d_{\Omega}\|_{L^{\infty}(\Omega\setminus B_{R_{\epsilon}} )} <\frac{\varepsilon}{2\,M}.
	\]
	By using the properties of $\eta$ and applying H\"older's and Hardy's inequalities to $u_{p_n} \eta_{R_{\varepsilon}}\in W^{1,p}_0(\Omega)$, for every $p_n \ge \max\{\overline{p},p_{n_0}\}$, we have that 
	\[
	\begin{split} 
		\int_{\Omega\setminus B_{R_{\varepsilon}+1}} |u_{{p_n}}|^{p_n} \, dx =\int_{\Omega \setminus B_{R_{\varepsilon}+1}} |u_{p_n} \eta_{R_{\varepsilon}}|^{p_n} \, dx &\le \|d_{\Omega}\|^{\,p_n}_{L^{\infty}(\Omega\setminus B_{R_{\varepsilon}})} \int_{\Omega} \frac{ |u_{p_n}\, \eta_{R_{\epsilon}}|^{p_n}}{d_{\Omega}^{\,p_n}}\, dx \\ 
		&\le\|d_{\Omega}\|^{\,p_n}_{L^{\infty}(\Omega\setminus B_{R_{\varepsilon}})} \, \frac{1}{\mathfrak{h}_{p_n}(\Omega)} \int_{\Omega} |\nabla (u_{p_n} \,\eta_{R_{\varepsilon}}) |^{p_n}\, dx \\
		&\le \varepsilon^{\,p_n}.
	\end{split}
	\]
	Hence \[1=\int_{\Omega\setminus B_{R_\varepsilon+1}} |u_{p_n}|^{p_n} \, dx + \int_{B_{R_\varepsilon+1}} |u_{p_n}|^{p_n} \, dx \le  \varepsilon^{\,p_n}+ \int_{B_{R_\varepsilon+1}} |u_{p_n}|^{p_n} \, dx,\]
	that is \[ \big(1-\varepsilon^{\,p_n}\big)^{\frac{1}{p_n}} \le \Big(\omega_N \,(R_\epsilon+1)^N \Big)^{\frac{1}{p_n}} \sup_{B_{R_\epsilon+1}}|u_{p_n}|, \qquad \mbox{for every } p_n \ge \max\{\overline{p},p_{n_0}\}.
	\]
	By exploiting the uniform convergence of the family $\{u_{p_n}\}_{n\ge n_0}$ to $u_{\infty}$ on compact sets,  if we take the limit as $n$ goes to $\infty$, we get
	\[ 1 \le \sup_{B_{R_{\varepsilon}+1}}|u_{\infty}|\le \sup_{\Omega}|u_{\infty}|, \]
	which proves the claim for every $\Omega$ quasibounded set.
	\par
	In order to get the convergence in $C^{0,\beta}(\overline\Omega)$ for $0<\beta<1$, we can use the same interpolation argument as in {\it Part 2} of the proof of Theorem \ref{teo:asymp}. It is sufficient to observe that it holds
	\[
	[u_\infty]_{C^{0,\beta}(\overline\Omega)}\le 2^{1-\beta}\,\|u_\infty\|^{1-\beta}_{L^\infty(\Omega)}\,\|\nabla u_\infty\|^\beta_{L^\infty(\Omega)},
	\]
	an estimate that can be proved by repeating the proof of Lemma \ref{lm:distholder}. We leave the details to the reader.
\end{proof}
 \begin{rem}  
			We underline that, differently from the case $1\le q<p$, when $p=q$ it may happen that the accumulation points of the family $\{u_p\}_{p>N}$ do not coincide with $d_{\Omega}$ (see \cite[Corollary 4.7]{FIN}). We refer to \cite[Theorem 3.14]{EP} for a study of the multiplicity of extremals for problem \eqref{mininfty}, in the case of open bounded sets.
		\end{rem}
With the same arguments as in the previous proof, we  can show a similar result for minimizers of $\lambda_{p,\infty}(\Omega)$, whose existence is given by Corollary \ref{coro:endpoint}  when $\Omega$ is quasibounded. This generalizes   \cite[Theorem 3.3]{EP}. We omit the proof.
\begin{thm}\label{teo:asympwpinfty}
	Let $N< p $ and let $\Omega \subsetneq \mathbb{R}^N$ be a quasibounded open set. Let $u_{p,\infty} \in W^{1,p}_0(\Omega)$ be a positive solution of 
	\[ 	\lambda_{p,\infty}(\Omega)=\min_{u\in W^{1,p}_0(\Omega)} \left\{\int_\Omega |\nabla u|^p\,dx\, :\, \|u\|_{L^\infty(\Omega)}=1\right\}. \]
	Then the family $\{ u_{p,\infty}\}_{p>N}$ is precompact in $C^{0,\beta}(\overline\Omega)$ for every $0<\beta<1$ and every accumulation point $u_{\infty}$ is a solution of \eqref{mininfty}.	
\end{thm}

	\appendix
	
	\section{An infinite strip with slowly shrinking ends}
	\label{sec:A}
	
	In the next example, we consider a quasibounded open set for which $d^\gamma_\Omega\notin L^1(\mathbb{R}^N)$, for any $0<\gamma<\infty$. 
	
	\begin{ex}\label{esempio}
	 For every $\alpha>0$ and $x_1 \in \mathbb{R}$, we set 
		\[
	f_1(x_1)=\frac{1}{\log \left(2+x_1^2\right)}\qquad \mbox{ and }\qquad	f_\alpha(x_1)=f_1\left(\frac{x_1}{\alpha}\right)=\frac{1}{\log \left(2+\left(\dfrac{x_1}{\alpha}\right)^2\right)}.
		\] 
	Then we consider the quasibounded open set
		\[
		\Omega_\alpha=\Big\{x=(x_1,x_2)\in\mathbb{R}^2\, :\, x_1\in\mathbb{R},\, |x_2|< f_\alpha(x_1) \Big\}, \qquad \mbox{ with } \alpha^2 > (\log 2)^{-3}.
		\]
		Observe that for this set we have $d_\Omega^\gamma\not \in L^1(\Omega)$, for any $0<\gamma<\infty$. Thus, by Theorem \ref{teo:q<p} part (i), we have
		\[
		\mathcal{D}^{1,2}_0(\Omega_\alpha)\not\hookrightarrow L^q(\Omega_\alpha),\qquad \mbox{ for every } 1\le q<2.
		\]
		On the other hand, since  $\Omega_\alpha$ is bounded in the $x_2$ direction, we easily get that $\lambda_{2}(\Omega_\alpha)>0$,
		that is 
		\[
		\mathcal{D}^{1,2}_0(\Omega_\alpha)\hookrightarrow L^2(\Omega_\alpha).
		\]
	As for the compactness of this embedding, we observe that this can not be directly inferred from Theorem \ref{teo:q=p}, since we are in the critical situation $p=2=N$. Nevertheless, we are going to show that actually such an embedding is compact, thanks to the particular geometry of the set $\Omega_{	\alpha}$. In particular, the Dirichlet-Laplacian on $\Omega_\alpha$ has a discrete spectrum.
\par
We define 
		\[
		\Omega_{\alpha,R} := \Omega_{\alpha} \cap \Big((-R, R)\times (-R,R)\Big),\qquad \mbox{ for } R\ge R_0=\frac{1}{\log 2}.
		\] 
		We denote by $w_{\Omega_\alpha}$ the torsion function of $\Omega_\alpha$, defined as 
		\[ w_{\Omega_\alpha} := \lim_{R \to \infty} w_{\Omega_{\alpha,R}}, \]
		where $w_{\Omega_{\alpha,R}} \in W^{1,2}_0(\Omega_{\alpha,R})$ is the {\it torsion function} of $\Omega_{\alpha,R}$, i.\,e. it solves
		\begin{equation}\label{eq:torsione}
			-\Delta u = 1, \quad \mbox{ in } \Omega_{\alpha,R},
		\end{equation}
		(see \cite[Definition 2.2]{BR}).
		In order to prove the compactness of the embedding of $\mathcal{D}^{1,2}_0(\Omega_\alpha)\hookrightarrow L^2(\Omega_\alpha)$, it is sufficient to prove that
		\begin{equation}\label{eq:claim} 
			\lim_{R \to \infty} \|w_{\Omega_{\alpha}}\|_{L^\infty (\Omega\setminus B_R)} = 0,
		\end{equation}
		thanks to \cite[Theorem 1.3]{BR}. We will achieve  \eqref{eq:claim} by exploiting the geometry of $\Omega_\alpha$ in order to construct a suitable upper barrier.
		 For every $\alpha>0$ and $x_1 \in \mathbb{R}$, we set 
		\[
	F_1(x_1)=(f_1(x_1))^2\qquad \mbox{ and }\qquad	F_{\alpha}(x_1) := F_1\left(\frac{x_1}{\alpha}\right)= (f_{\alpha}(x_1))^2.
		\]
		Observe that
		\[ 
		\left|F_1''\left(\frac{x_1}{\alpha}\right)\right| \le 2 \,(\log 2)^{-3},
		\]
	thus, if we take $\alpha>0$ such that $\alpha^2 > (\log 2)^{-3}$, we obtain
				\begin{equation}\label{derivsec} 
			|F_{\alpha}''(x_1)| = \frac{1}{\alpha^2}\,\left|F_1''\left(\frac{x_1}{\alpha}\right)\right| \le 2\,(1-C), \qquad \mbox{ for } x_1\in \mathbb{R},
		\end{equation}
		for some $C=C(\alpha)\in (0,1)$. With such a choice of $\alpha$, we consider the function
		\[ 
		U_\alpha(x_1,x_2)=\frac{F_\alpha(x_1)-x_2^2}{2\,C}, \quad \mbox{ for every } (x_1,x_2) \in \Omega_{\alpha}.
		\]
		We claim that this is the desired upper barrier. Indeed, by construction we have $U_{\alpha} \ge 0$ and, thanks to \eqref{derivsec}, it holds
		\[ 
		-\Delta U_\alpha (x_1,x_2)= \frac{1}{C}-\frac{F''_\alpha(x_1)}{2\,C} \ge 1 , \quad \mbox{ for every } (x_1,x_2) \in \Omega_{\alpha}. 
		\]
	By applying the Comparison Principle in every $\Omega_{\alpha,R}$  
	we get that
		\[ 
		w_{\Omega_{\alpha,R}}(x) \le U_\alpha(x) \le \frac{F_\alpha(x_1)}{2\,C}, \quad \mbox{ for every } x=(x_1,x_2) \in \Omega_{\alpha,R}, 
		\]
		and such an estimate does not depend on $R$. Hence, by sending $R$ to $\infty$, we have that 
		\[ 
		w_{\Omega_\alpha}(x) \le \frac{F_\alpha(x_1)}{2\,C}, \quad \mbox{ for every } x=(x_1,x_2) \in \Omega_\alpha.
		\]
		By using the properties of $F_\alpha=(f_\alpha)^2$ and the previous estimate, we eventually get \eqref{eq:claim}.  
	\end{ex}

\end{document}